\documentclass[a4paper,11pt]{amsart}
\usepackage[utf8]{inputenc}
\usepackage{amsfonts}
\usepackage{hyperref}
\usepackage{amsmath}
\usepackage{amsthm}
\usepackage[capitalize,nameinlink]{cleveref}
\usepackage{comment}
\usepackage{mathrsfs}
\usepackage{mathabx}
\usepackage{tikz}
\usepackage[disable]{todonotes}
\usetikzlibrary{positioning}
\usepackage{caption}
\newtheorem{theorem}{Theorem}[section]
\newtheorem{corollary}[theorem]{Corollary}
\newtheorem{lemma}[theorem]{Lemma}
\newtheorem{proposition}[theorem]{Proposition}
\newtheorem{remark}[theorem]{Remark}
\newtheorem{definition}[theorem]{Definition}
\newtheorem{example}[theorem]{Example}
\newtheorem{claim}[theorem]{Claim}
\DeclareMathOperator*{\esssup}{ess\,sup}

\newcommand{\Addresses}{{
    \bigskip
    \footnotesize
    
    Monika Dörfler, \textsc{Department of Mathematics, University of Vienna, 1090 Vienna, Austria}\par\nopagebreak
    \textit{E-mail address}: \texttt{monika.doerfler@univie.ac.at}
    
    \medskip
    
    Franz Luef, \textsc{Department of Mathematical Sciences, Norwegian University of Science and Technology, 7034
    Trondheim, Norway}\par\nopagebreak
    \textit{E-mail address}: \texttt{franz.luef@ntnu.no}
    
    \medskip
    
    Henry McNulty, \textsc{Cognite AS, 1366 Lysaker, Norway}\par\nopagebreak
    \textsc{Department of Mathematical Sciences, Norwegian University of Science and Technology, 7491
    Trondheim, Norway}\par\nopagebreak
    \textit{E-mail address}: \texttt{henry.mcnulty@cognite.com}
    
    \medskip
    
    Eirik Skrettingland, \textit{E-mail address}: \texttt{skrettingland.eirik@gmail.com}

}}

\title{Time-Frequency Analysis and Coorbit Spaces of Operators}
\author{
  Monika Dörfler
  \and Franz Luef
  \and Henry McNulty 
  \and Eirik Skrettingland
  }
\begin{document}

\maketitle
\begin{abstract}
    We introduce an operator valued Short-Time Fourier Transform for certain classes of operators with operator windows, and show that the transform acts in an analogous way to the Short-Time Fourier Transform for functions, in particular giving rise to a family of vector-valued reproducing kernel Banach spaces, the so called coorbit spaces, as spaces of operators. As a result of this structure the operators generating equivalent norms on the function modulation spaces are fully classified. We show that these operator spaces have the same atomic decomposition properties as the function spaces, and use this to give a characterisation of the spaces using localisation operators.
\end{abstract}

\section{Introduction}
In time-frequency analysis, the modulation spaces $M^{p,q}_m(\mathbb{R}^d)$, first introduced by Feichtinger in 1983 \cite{feich83}, play a central role, where they define spaces of functions with certain desirable time-frequency decay. In particular the Feichtinger algebra, $M^1(\mathbb{R}^d)$ or $\mathbf{S}_0(\mathbb{R}^d)$ \cite{feich81} \cite{feich79}, gives well concentrated functions in the time-frequency sense, which are for many purposes the ideal atoms for Gabor analysis. The modulation spaces are usually defined in terms of the Short-Time Fourier Transform (STFT), namely as the spaces
\begin{align*}
    M^{p,q}_m(\mathbb{R}^{d}) := \{\psi\in\mathcal{S}'(\mathbb{R}^d): \Big( \int_{\mathbb{R}^d}\Big(\int_{\mathbb{R}^d} |V_{\varphi_0} \psi(z)|^p m(x,\omega)^p dx\Big)^{q/p}d\omega\Big)^{1/q} < \infty \},
\end{align*}
where $\varphi_0$ is the Gaussian. Modulation spaces and their various generalisations have been studied extensively, and surveys and monographs can be found in \cite{jakob18} \cite{benyimodspaces} \cite{grochenigtfa}. The properties and utility of these function spaces are too broad to hope to cover, but of particular interest to our work is that these spaces are the coorbit spaces \cite{feich89i} \cite{feich89ii} of the projective unitary representation of the reduced Weyl-Heisenberg group, and as such have (among others) the following properties:
\begin{enumerate}
    \item All $g\in L^2(\mathbb{R}^d)$ that satisfy the condition $V_g g \in L^1_v(\mathbb{R}^{2d})$, generate the same modulation spaces $M^{p,q}_m(\mathbb{R}^{d})$ as windows, and their norms are equivalent.
    \item (\textit{Correspondence Principle}) Given an atom $g$ as above, there is an isometric isomorphism $M^{p,q}_m(\mathbb{R}^{d}) \cong \{F\in L^{p,q}_m(\mathbb{R}^{2d}): F = F \natural V_g g\}$ (where $\natural$ is the twisted convolution discussed below), given by $V_g$. Note that the later are reproducing kernel Banach spaces.
\end{enumerate}
There is a vast body of contributions to the theory of coorbit spaces, e.g. \cite{bagr17,chol11,hovo21}.
In this work we examine spaces of operators exhibiting similar properties, by introducing an STFT with operator window and argument, returning an operator-valued function on phase space. One motivation comes from \cite{Dorf21}, where local structures of a data set $\mathcal{D}=\{f_1,...,f_N\}$ were identified via mapping the data points of functions $f_i$ on $\mathbb{R}^d$ to rank-one operators $f_i\otimes f_i$, and constructing the data operator $S_{\mathcal{D}}=\sum_{i=1}^N f_i\otimes f_i$. Hence, it would be of interest to compare to data sets $\mathcal{D}$ and $\mathcal{D}^\prime$ via its respective data operators $S_\mathcal{D}$ and $S_{\mathcal{D}^\prime}$. Another source of inspiration is the work \cite{keyl16}, where operator analogues of the Schwartz class of functions and of the space of tempered distributions have been introduced and their basic theory has been developed along the lines of the function/distribution case. 

The concept of an STFT for operators is not a new one. In \cite{balasz2019}, the authors consider the wavelet transform for the representation $\pi(w)\otimes\pi(z)$ on $\mathcal{HS}=L^2(\mathbb{R}^d)\otimes L^2(\mathbb{R}^d)$ to examine kernel theorems for coorbit spaces. This entails using the standard scalar-valued construction for the coorbit spaces defined by the wavelets transform, giving different spaces to our approach. On the other hand nor are vector-valued reproducing kernel Hilbert spaces in time-frequency analysis a new concept. In \cite{balan00} and \cite{abreau10} an STFT is constructed for vectors of functions, which results in a direct sum of Gabor spaces. Our work differs from these in that windows, arguments and resulting output of the operator STFT are all operators. 

In \cite{skrett22}, the author introduced an equivalent notion of a STFT with an operator window, given by 
\begin{align} \label{functionstft}
    \mathfrak{V}_S \psi(z) := S\pi(z)^* \psi
\end{align}
for some appropriate operator $S$ and function $\psi$. In particular, the question was considered of which operators would define equivalent norms on $M^{p,q}_m(\mathbb{R}^{d})$ under this STFT, that is, for which operators
\begin{align*}
    \|\psi\|_{M^{p,q}_m} \asymp  \|S\pi(z)^*\psi\|_{L^{p,q}_m(\mathbb{R}^{2d};L^2)}.
\end{align*}
In further work by Guo and Zhao \cite{Guo22}, some equivalent conditions for equivalence were given. In both works a class of operators with adjoints in a certain class of nuclear operators was discussed, along with the open question in the latter of whether these operators exhausted all possible operators generating equivalent norms on $M^{p,q}_m(\mathbb{R}^{d})$. In this work we present an extension of the operator window STFT \eqref{functionstft}, which acts on operators instead of functions. We initially define such a transform for $S,T\in\mathcal{HS}$ in the following manner:
\begin{definition}{(Operator STFT)}
For $S,T\in\mathcal{HS}$, the STFT of $T$ with window $S$, is given by
\begin{align}
    \mathfrak{V}_S T(z) := S^*\pi(z)^* T.
\end{align}
\end{definition}
Note that in the case of rank-one operators $S=g\otimes e$ and $T=f\otimes e$ for $e,f,g\in L^2(\mathbb{R}^d)$ the operator STFT becomes $V_gf(z)e\otimes e$, which is the STFT of functions embedded into the space of Hilbert-Schmidt operator-valued functions.

We examine the behaviour of this transform, e.g. Moyal's identity, paying particular attention to the spaces it produces as images. In this respect the first result of this paper demonstrates a parallel to the STFT of functions, regarding the reproducing structure of the image of the Hilbert space of Hilbert-Schmidt operators:\todo{M: maybe make new own contributions more explicit}
\begin{theorem}
    For any Hilbert-Schmidt operator $S$, the space defined by
    \begin{align*}
        \mathfrak{V}_S (\mathcal{HS}) := \{\mathfrak{V}_S T(z): T\in \mathcal{HS}\}
    \end{align*}
    is a vector-valued uniform reproducing kernel Hilbert space as a subspace of the Bochner-Lebesgue space $L^2(\mathbb{R}^{2d};\mathcal{HS})$.
\end{theorem}
Motivated by this, we extend the reproducing properties of this space to the "coorbit spaces", and consider the spaces $\mathfrak{A}_v := \{S\in \mathcal{HS}: \mathfrak{V}_S S \in L^1_v(\mathbb{R}^{2d};\mathcal{HS})\}$, and \todo{M: when using not so popular notion refer to section where they will be defined?}$\mathfrak{M}^{p,q}_m := \{T\in \mathfrak{S}': \mathfrak{V}_{S} T \in L^{p,q}_m(\mathbb{R}^{2d};\mathcal{HS})\}$, where $\mathfrak{S}'$ are operators with Weyl symbols in $\mathscr{S}^\prime$ and $S\in\mathfrak{A}_v$, to derive the result
\begin{theorem}
    For any $S\in \mathfrak{M}^1_v$, we have an isometric isomorphism
    \begin{align*}
        \mathfrak{M}^{p,q}_m \cong \{\Psi\in L^{p,q}_m(\mathbb{R}^{2d};\mathcal{HS}): \Psi = \Psi \natural \mathfrak{V}_S S\}
    \end{align*}
    under the mapping
    \begin{align*}
        T \mapsto \mathfrak{V}_S T,
    \end{align*}
\end{theorem}
 the twisted convolution $\natural$ is to be defined in \cref{rkhs}. Furthermore, for all $S\in\mathfrak{A}_v$ the resulting spaces coincide, and the associated norms are equivalent. The dual space of $\mathfrak{M}^{p,q}_m$ is $\mathfrak{M}^{p',q'}_{1/m}$, where $\frac{1}{p}+\frac{1}{p'}=1$, $\frac{1}{q}+\frac{1}{q'}=1$ with the usual adjustment for $p,q=1,\infty$. As a corollary of the coorbit structure and independence of windows, we characterise operators satisfying the equivalent norm condition;
\begin{corollary}
    The operators which define equivalent norms on the spaces $M^{p,q}_m(\mathbb{R}^{d})$ by
    \begin{align*}
        \|S^*\pi(z)^*\psi\|_{L^{p,q}_m(\mathbb{R}^{2d};L^2(\mathbb{R}^d))}
    \end{align*}
    or every $1\leq p,q \leq \infty$ and $v$-multiplicative $m$, are precisely the admissible operators 
    \begin{align*}
        \mathfrak{A}_v := \{S: \mathfrak{V}_S S \in L^1_v(\mathbb{R}^{2d};\mathcal{HS})\},
    \end{align*}
\end{corollary}
We finally consider the atomic decomposition of operators in the $\mathfrak{M}^{p,q}_m$, which follows from the same arguments as the function case given the coorbit structure. Using this we can characterise the spaces using localisation operators:
\begin{corollary}
    Let $\varphi\in L^2(\mathbb{R}^d)$ be non-zero and $h\in L^1_v(\mathbb{R}^{2d})$ be some non-negative symbol satisfying
    \begin{align*}
        A \leq \sum_{\lambda\in\Lambda} h(z-\lambda) \leq B
    \end{align*}
    for positive constants $A,B$, and almost all $z\in\mathbb{R}^{2d}$. Then for every $v$-moderate weight $m$ and $1\leq p < \infty$ the operator $T\in\mathfrak{M}^{\infty}_{1/v}$ belongs to $\mathfrak{M}^{p,q}_m$ if and only if 
    \begin{align*}
        \big\{ A_{\overline{h}}^{\varphi} \pi(\lambda)^* T \big\}_{\lambda\in\Lambda} \in l^{p,q}_m(\Lambda;\mathcal{HS}).
    \end{align*}
    where $\Lambda = \alpha\mathbb{Z} \times \beta\mathbb{Z}$ is some full rank lattice.
\end{corollary}

\section{Preliminaries}
\subsection{Time-Frequency Analysis Basics}
While coorbit spaces are defined in general for integrable representations of locally compact groups, modulation spaces of functions and the spaces discussed in this work arise from the particular case of the time-frequency shifts $\pi(z)$, the projective unitary representation of the reduced Weyl-Heisenberg group on the Hilbert space $L^2(\mathbb{R}^d)$. Such shifts can be defined as the composition of the translation operator $T_x:f(t)\mapsto f(t-x)$, and the modulation operator $M_{\omega}:f(t)\mapsto e^{2\pi i \omega t}f(t)$, by the identity
\begin{align*}
    \pi(z) = M_{\omega}T_x
\end{align*}
where $z=(x,\omega)\in \mathbb{R}^{2d}$. Direct calculations show that $\pi(z)$ is unitary on $L^2(\mathbb{R}^d)$, and that we have
\begin{align*}
    \pi(z)\pi(z') &= e^{-2\pi i \omega' x}\pi(z+z') \\
    \pi(z)^* &= e^{-2\pi i x \omega}\pi(-z).
\end{align*}
The Short-Time Fourier Transform (STFT) for functions is then defined, for two functions $f,g\in L^2(\mathbb{R}^d)$, by
\begin{align}
    V_g f(z) := \langle f, \pi(z) g\rangle_{L^2}.
\end{align}
The window function $g$ is usually chosen to have compact support, or be concentrated around the origin, such as in the case of the normalised Gaussian $\varphi_0(t)=2^{d/4}e^{\pi t^2}$. For $f,g\in L^2(\mathbb{R}^{2d})$, $V_g f$ is uniformly continuous as a function in $L^2(\mathbb{R}^{2d})$, which will be instructive when considering reproducing kernel Hilbert spaces later. One has for the STFT \textit{Moyal's Identity} (see for example Theorem 3.2.1 of \cite{grochenigtfa}), giving an understanding of the basic properties of the STFT in terms of its window:
\begin{lemma}{(Moyal's Identity)}
Given functions $f_1,f_2,g_1,g_2\in L^2(\mathbb{R}^d)$, we have $V_{g_1} f_1,\,V_{g_2} f_2\in L^2(\mathbb{R}^{2d})$, and in addition:
\begin{align*}
    \langle V_{g_1} f_1, V_{g_2} f_2 \rangle_{L^2(\mathbb{R}^{2d})} = \langle f_1, f_2 \rangle_{L^2(\mathbb{R}^d)} \overline{\langle g_1, g_2 \rangle_{L^2(\mathbb{R}^d)}}.
\end{align*}
\end{lemma}
As a direct consequence, we have that for any $g \in L^2(\mathbb{R}^d)$ such that $\|g\|_{L^2}=1$, the map $V_g: L^2(\mathbb{R}^d)\to L^2(\mathbb{R}^{2d})$ is an isometry. As such, we can consider the inverse mapping. Rearranging Moyal's identity shows the reconstruction formula 
\begin{align}
    f = \int_{\mathbb{R}^{2d}} V_g f(z) \pi(z)g\, dz,
\end{align}
for any $g\in L^2(\mathbb{R}^d)$ with $\|g\|=1$. A direct calculation then shows that the adjoint $V_g^*$ is given by 
\begin{align}
    V_g^*(F) :=  \int_{\mathbb{R}^{2d}} F(z) \pi(z)g\, dz,
\end{align}
where the integral can be interpreted in the weak sense, and so from the reconstruction formula
\begin{align*}
    V_g^* V_g = I_{L^2(\mathbb{R}^d)}.
\end{align*}

\subsection{Weight functions and mixed-norm spaces}
We begin by defining a sub-multiplicative weight $v$ as a non-negative, locally integrable function on phase space $\mathbb{R}^{2d}$ satisfying the condition
\begin{align*}
    v(z_1+z_2) \leq v(z_1)v(z_2)
\end{align*}
for all $z_1,z_2\in\mathbb{R}^{2d}$. As a direct result, $v(0)\geq 1$. A $v$-moderate weight $m$ is then a non-negative, locally integrable function on phase space such that 
\begin{align*}
    m(z_1 + z_2) \leq v(z_1)m(z_2)
\end{align*}
for all $z_1,z_2\in\mathbb{R}^{2d}$. As a particular consequence, we have for such a $v,m$ that 
\begin{align*}
    \frac{1}{C_{v,m}v(z)} \leq m(z) \leq C_{v,m}v(z).
\end{align*}
In this work we consider weights of at most polynomial growth. We define the weighted, mixed-norm space $L^{p,q}_m(\mathbb{R}^{2d})$, for $1\leq p,q < \infty$, as the functions for which the norm
\begin{align*}
    \|F\|_{L^{p,q}_m} := \Big(\int_{\mathbb{R}^d}\Big(\int_{\mathbb{R}^d} |F(x,\omega)|^p m(x,\omega)^p\, dx\Big)^{q/p}\, d\omega \Big)^{1/q}
\end{align*}
is finite. In the case where $p$ or $q$ is infinite, we replace the corresponding integral with essential supremum. For such spaces we have the duality $(L^{p,q}_m(\mathbb{R}^{2d}))' = L^{p',q'}_{1/m}(\mathbb{R}^{2d})$, where $\frac{1}{p}+\frac{1}{p'} = 1$, $\frac{1}{q}+\frac{1}{q'} = 1$. Further details on weights and mixed-norm spaces can be found in chapter 11, \cite{grochenigtfa}. In this work we consider discretisation over the full rank lattice $\Lambda = \alpha\mathbb{Z}^{d}\times \beta\mathbb{Z}^{d}$. An arbitrary lattice $\Lambda = A\mathbb{Z}^{2d}$, $A\in GL(2d;\mathbb{R})$ can also be used, but the notion of mixed-norm becomes less clear. We define the mixed-norm weighted sequence space $l^{p,q}_m(\Lambda;\mathcal{HS})$ as the sequences $a_{(k,l)}$ such that 
\begin{align*}
    \|a\|_{l^{p,q}_m(\Lambda;\mathcal{HS})} := \Big(\sum_{n\in\mathbb{Z}^d} \big(\sum_{k\in\mathbb{Z}^d} m(\alpha k,\beta l)^p \|a_{\alpha k,\beta l}\|_{\mathcal{HS}}^p \big)^{q/p} \Big)^{1/q} < \infty.
\end{align*}
The Wiener Amalgam spaces introduced in \cite{feich83i} provide the required framework for sampling estimates on the lattice. To that end we define for a given function $\Psi: \mathbb{R}^{2d}\to\mathcal{HS}$ the sequence
\begin{align*}
     a^{\Psi}_{(k,l)} = \Big(\esssup_{x,\omega\in [0,1]^d} \|\Psi(x+k,\omega+l)\|_{\mathcal{HS}}\Big)_{(k,l)}.
\end{align*}
\begin{definition}\label{Wieneramalg}
    Let $1\leq p,q \leq \infty$ and $m$ be some weight function. The Wiener Amalgam space $W(L^{p,g}_m(\mathbb{R}^{2d};\mathcal{HS}))$ consists of all functions $\Psi: \mathbb{R}^{2d}\to\mathcal{HS}$ such that
    \begin{align*}
        \|a^{\Psi}_{(k,l)}\|_{l^{p,q}_m} < \infty,
    \end{align*}
    with the norm $\|\Psi\|_{W(L^{p,q}_m(\mathbb{R}^{2d};\mathcal{HS}))} := \|a^{\Psi}_{(k,l)}\|_{l^{p,q}_m}$.
\end{definition}
One feature of the Wiener Amalgam spaces we use (see for example Proposition 11.1.4 of \cite{grochenigtfa}) is the following:
\begin{proposition}\label{wieneramalg}
Let $\Lambda = \alpha\mathbb{Z}^{d}\times \beta\mathbb{Z}^{d}$ and $\Psi\in W(L^{p,q}_m(\mathbb{R}^{2d};\mathcal{HS}))$ be continuous. Then
\begin{align*}
    \|\Psi|_{\Lambda}\|_{l^{p,q}_{\Tilde{m}}(\Lambda;\mathcal{HS})} \leq c\|\Psi\|_{W(L^{p,q}_m(\mathbb{R}^{2d};\mathcal{HS}))}
\end{align*}
where $\Tilde{m}(k,l)=m(\alpha k, \beta l)$, and $c$ depends on the lattice $\Lambda$.
\end{proposition}
While stated for scalar-valued functions in \cite{grochenigtfa}, the same argument gives the vector-valued case. 

\subsection{Reproducing Kernel Hilbert Spaces}\label{rkhs}
\subsubsection{Vector-Valued RKHS}
We recall definitions and identities in this section in terms of \textit{vector-valued} reproducing kernel Hilbert spaces, following the formalism of Paulsen and Raghupathi in chapter 6 of \cite{paulsenrkhs}. The familiar scalar case follows simply by considering the vector space which functions take their values to be $\mathbb{C}$.
\begin{definition}{}
Let $\mathcal{C}$ be a Hilbert space, and $X$ some set. We denote by
$\mathcal{F}(X, \mathcal{C})$ the vector space of $\mathcal{C}$-valued functions under the usual pointwise sum and scalar multiplication. A subspace $\mathcal{H}\subseteq \mathcal{F}$ is a $\mathcal{C}$-valued reproducing Kernel Hilbert Space (RKHS) if it is a Hilbert space, and for every $x\in X$, the evaluation map $E_x:f\to f(x)$ is a bounded operator. If the set $\{E_x\}_{x\in X}$ is uniformly bounded in norm, then $\mathcal{H}$ is referred to as \textit{uniform}.
\end{definition}
Since $\mathcal{H}$ is a Hilbert space, it follows from Riesz' representation theorem that for each $E_x$, there is some $k_x\in \mathcal{H}$ such that $E_x(f) = \langle f, k_x\rangle_{\mathcal{H}}$. It follows from definition that 
\begin{align*}
    |f(x)-g(x)| = |\langle f, k_x\rangle_{\mathcal{H}} - \langle g, k_x\rangle_{\mathcal{H}}| = |\langle f - g, k_x\rangle_{\mathcal{H}}| \leq \|f-g\|_{\mathcal{H}} \|k_x\|_{\mathcal{H}},
\end{align*}
so unlike in the general Hilbert space setting, we have pointwise bounds in terms of norms in the RKHS setting. The \textit{kernel function} $K:X\times X \to \mathcal{L}(\mathcal{C})$ is defined as $K(x,y) = E_x E_y^*$, and has the property $K(x,y) = K(x,y)^*$. The kernel function uniquely defines the RKHS, that is to say given two RKHS' $\mathcal{H}_1,\mathcal{H}_2$, if $K_1(x,y)=K_2(x,y)$ then $\mathcal{H}_1 = \mathcal{H}_2$ and $\|\cdot\|_{\mathcal{H}_1} = \|\cdot \|_{\mathcal{H}_2}$, and vice versa.
\begin{example}
\normalfont Given $g\in L^2(\mathbb{R}^d)$ such that $\|g\|_{L^2}=1$, the \textit{Gabor space} $V_g(L^2(\mathbb{R}^d)) \subset L^2(\mathbb{R}^{2d})$ with norm $\|V_g f\|_{V_g(L^2)}=\|f\|_{L^2}$ is a RKHS with kernel $K(z,z') = \langle \pi(z')g, \pi(z) g\rangle$. 
\end{example}

This result can be deduced by noting that $V_g$ is an isometry onto its image, then proceeding with the adjoint as defined above. 

\subsubsection{Twisted Convolutions}
Reproducing properties of Gabor spaces are intimately connected to the \textit{twisted convolution}, which is defined in terms of the $2$-cocycle of $\pi(z)$, which we define as $c(z,z')=e^{-2\pi ix'(\omega-\omega')}$, such that $\pi(z)^*\pi(z')=c(z,z')\pi(z+z')^*$. We define the twisted convolution for a Lebesgue-Bochner space (for details see for example \cite{DincVecMeas}). The twisted convolution presented here swaps the arguments, but this is only in order to fit our construction of the operator STFT with the standard notation in coorbit theory.
\begin{definition}
Given two operator-valued functions $F,H\in L^2(\mathbb{R}^{2d};\mathcal{HS})$, we define the twisted convolution $\natural$ as
\begin{align*}
    F\natural H(x) = \int_{\mathbb{R}^{2d}} H(x-y)F(y)c(x-y,y)\, dy,
\end{align*}
where the integral can be interpreted in the sense of a Bochner integral.
\end{definition}
In the concrete setting of Gabor spaces, a direct calculation shows for functions $f_1,f_2,g_1,g_2\in L^2(\mathbb{R}^{2d})$, that $V_{g_1} f_1 \natural V_{g_2} f_2 = \langle f_2, g_1 \rangle V_{g_2} f_1$. Clearly then for some $F\in V_g(L^2)$, the identity $F\natural V_g g = F$ holds when $\|g\|_{L^2}=1$, however a fundamental result of coorbit theory is that the converse also holds, giving the following:
\begin{proposition} \label{rkhsclassification}
    Given some $g\in L^2(\mathbb{R}^d)$ with $\|g\|_{L^2}=1$, a function $F\in L^2(\mathbb{R}^{2d})$ is in $V_g (L^2)$ if and only if $F\natural V_g g = F$. 
\end{proposition}
We note an application of weighted, mixed-norm Young's inequality to Lebesgue-Bochner spaces of Banach algebras to be used in the sequel. A proof of the scalar valued case can be found for example in Proposition 11.1.3 of \cite{grochenigtfa}, the vector valued case follows by the same argument. 
\begin{lemma} \label{youngsineq}
Given functions $F\in L^1_v(\mathbb{R}^{2d};\mathcal{HS})$ and $H\in L^{p,q}_m(\mathbb{R}^{2d};\mathcal{HS})$ we have
\begin{align*}
     \|F\natural H \|_{L^{p,q}_m} \leq C_{m,v}\|F\|_{L^1_v} \|H\|_{L^{p,q}_m},
\end{align*}
where $v$ is some sub-multiplicative function and $m$ a $v$-moderate weight, and $C_{m,v}$ a constant depending on $v$ and $m$.
\end{lemma}

\subsection{Modulation Spaces}
We begin by considering the space $M^1_v(\mathbb{R}^{d})$. For a sub-multiplicative $v$, we define the modulation space as functions whose image under the STFT with Gaussian window is in $L^1_v(\mathbb{R}^{2d})$;
\begin{align*}
    M^1_v(\mathbb{R}^{d}) := \{f\in L^2(\mathbb{R}^d): V_{\varphi_0} f \in L^1_v(\mathbb{R}^{2d})\}.
\end{align*}
 Such a space is always non-empty, as $\varphi_0$ itself is contained in it, and for weights of polynomial growth it contains the Schwartz functions $\mathscr{S}$. In addition, it is closed under pointwise multiplication, time-frequency shifts, and is a Banach space under the norm $\|f\|_{M^1_v} = \|V_{\varphi_0} f\|_{L^1_v}$. The unweighted $M^1(\mathbb{R}^{d})$ is Feichtinger's algebra, which has been studied extensively and provides for many avenues of time-frequency analysis the ideal set of test functions. We refer to the early paper \cite{feich81} and recent survey \cite{jakob18} for more details on the space. General modulation spaces are then defined, for any $v$-moderate weight $m$, by
\begin{align*}
    M^{p,q}_m(\mathbb{R}^{d}) := \{f\in (M^1_v(\mathbb{R}^{d}))': V_{\varphi_0} f \in L^{p,q}_m(\mathbb{R}^{2d})\},
\end{align*}
with the associated norm $\|f\|_{M^{p,q}_m}=\|f\|_{L^{p,q}_m}$
For any $g\in M^1_v(\mathbb{R}^{d})$, the space $\{f\in (M^1_v(\mathbb{R}^{d}))': V_{g} f \in L^{p,q}_m(\mathbb{R}^{2d})\}$ is equal to the space $M^{p,q}_m(\mathbb{R}^{d})$ and the associated norms are equivalent. It is not hard to see that $M^{2,2}(\mathbb{R}^{d})=L^2(\mathbb{R}^{d})$, from the properties of the STFT with window $\varphi_0 \in L^2(\mathbb{R}^{2d})$. The modulation spaces form coorbit spaces of the unitary representation $\pi$, and as such have the property:
\begin{theorem}{(Correspondence Principle)}
Given some $g\in M^1_v(\mathbb{R}^{d})$, for every $1\leq p \leq \infty$, the STFT defines an isomorphism
\begin{align*}
    V_g: M^{p,q}_v(\mathbb{R}^{d}) \to \{F\in L^{p,q}_v(\mathbb{R}^{2d}): F \natural  V_g g = F\}.
\end{align*}
\end{theorem}
\begin{remark}
\normalfont In this work we consider the case of the operator STFT $\mathfrak{V}_S$. One might therefore ask why we do not refer simply to an operator modulation space. We believe this would be misleading, since the term modulation space refers to the construction of the spaces by the $M^1(\mathbb{R}^{d})$ condition $\int_{\hat{G}} \|M_\omega f * f\|_1 d\omega < \infty$. We do not work with the analogous concept of modulation for operators, so we choose to refer to them as coorbit spaces. Although not coorbit spaces in the "strict sense" \cite{voigt15}, since the elements are not in the representation space of $\pi$, they can be considered as generalised coorbit spaces in the sense that instead of our transforms being a functional, they are now simply maps between spaces of operators, for example $\mathcal{HS}\to\mathcal{HS}$. 
\end{remark}

\subsection{Spaces of Operators}
\subsubsection{Schatten class and nuclear operators}
In this work we consider several spaces of operators. We begin by defining the \textit{trace class} operators as
\begin{align*}
    \mathcal{S}^1:= \{T\in\mathcal{L}(L^2(\mathbb{R}^{d})): \sum_{n\in \mathbb{N}} \langle |T|e_n, e_n\rangle < \infty\}
\end{align*}
for any orthonormal basis $\{e_n\}_{n\in\mathbb{N}}$ of $L^2(\mathbb{R}^{d})$. For operators satisfying this condition, the sum $\sum_{n\in \mathbb{N}} \langle Te_n, e_n\rangle$ in fact converges for all orthonormal bases to the same value, the trace of $T$, given by tr$(T)$. The set of such operators is then a Banach space when equipped with the norm $\|T\|_{\mathcal{S}^1} = tr(|T|)$. The \textit{Hilbert-Schmidt} operators $\mathcal{HS}$, are the operators
\begin{align*}
    \mathcal{HS}:= \{T\in\mathcal{L}(L^2(\mathbb{R}^{d})): T^*T\in\mathcal{S}^1 \}.
\end{align*}
The space $\mathcal{HS}$ is a Hilbert-Schmidt space with the inner product $\langle S,T\rangle_{\mathcal{HS}} = \mathrm{tr}(ST^*)$, and contains $\mathcal{S}^1$ as a proper ideal. We will often use that every compact operator, and therefore every Hilbert-Schmidt and trace class operator, admits a spectral decomposition
\begin{align*}
    S = \sum_{n\in\mathbb{N}} \lambda_n \psi_n \otimes \phi_n,
\end{align*}
where $\lambda_n$ are the singular values of $S$, $\{\psi_n\}_{n\in \mathbb{N}}$ and $\{\phi_n\}_{n\in \mathbb{N}}$ are orthonormal sets and the sum converges in operator norm. Both these spaces are Banach algebras with their respective norms and two-sided ideals in $\mathcal{L}(L^2(\mathbb{R}^{d}))$, with $\mathcal{S}^1\subset\mathcal{HS}\subset \mathcal{L}(L^2(\mathbb{R}^{d}))$. The further \textit{Schatten class operators}, $\mathcal{S}^p$, are defined by the decay of their singular values;
\begin{align*}
    \mathcal{S}^p := \{T\in\mathcal{L}(L^2(\mathbb{R}^{d})): \{\lambda_n\}_{n\in\mathbb{N}} \in l^p \}
\end{align*}
where $\lambda_n$ are again the singular values of $T$. Clearly $\mathcal{HS}=\mathcal{S}^2$. We also 
introduce a space of \textit{nuclear operators}, a 
concept which generalises the concept of trace to 
operators between Banach spaces. In particular for 
two Banach spaces $X,Y$, the nuclear operators $\mathcal{N}(X,Y)$ are the linear operators 
$T$ which have an expansion 
\todo{M: $x_n$}
$T=\sum_n y_n\otimes x_n$, where $y_n\in Y$, $x_n\in X'$ such that $\sum_n \|y_n\|_Y\|x_n\|_{X'} < \infty$. These operators become a Banach space when endowed with the norm $\|T\|_{\mathcal{N}(X,Y)} = \inf \sum_n \|y_n\|_Y\|x_n\|_{X'}$ where the infimum is taken over all possible decompositions of $T$. In our case we are interested in the nuclear operators $\mathcal{N}(L^2(\mathbb{R}^{d});M^1_v(\mathbb{R}^{d}))$. Such operators may be defined as the projective tensor product $\mathcal{N}(L^2(\mathbb{R}^{d});M^1_v(\mathbb{R}^{d})) := M^1_v(\mathbb{R}^{d}) \Hat{\otimes}_{\pi} L^2(\mathbb{R}^{d})$, the completion of the algebraic tensor product $M^1_v(\mathbb{R}^{d}) \otimes L^2(\mathbb{R}^{d})$ with respect to the nuclear norm
\begin{align*}
    \|h\|_{M^1_v \otimes L^2} = \inf \left\{ \sum_{n=1}^N \|g_n\|_{M^1_v}\|f_n\|_{L^2}: h=\sum_{n=1}^N g_n \otimes f_n \right\}.
\end{align*}
Finally we introduce the \textit{Schwartz operators} $\mathfrak{S}$, as the space of bounded integral operators with kernel $k\in \mathscr{S}(\mathbb{R}^{2d})$. Such operators form a Frechet space as detailed in \cite{keyl16}, and the topological dual $\mathfrak{S}'$ consists of integral operators with kernels in $\mathscr{S}'(\mathbb{R}^{2d})$, which by the Schwartz kernel theorem is the space of operators from $\mathscr{S}(\mathbb{R}^{2d})$ to $\mathscr{S}'(\mathbb{R}^{2d})$. For polynomial sub-multiplicative weight $v$, we use the sequence of inclusions $\mathfrak{S}\subset \mathcal{N}(L^2(\mathbb{R}^{d});M^1_v(\mathbb{R}^{d})) \subset \mathcal{HS} \subset \mathfrak{S}'$.

\subsubsection{G-frames for Operators}
In the operator setting, we will consider g-frames as introduced in \cite{sun06} as an analogue to frames in the function setting. In particular, given a Hilbert space $\mathcal{U}$, and a sequence of Hilbert spaces $\{\mathcal{V}_i\}_{i\in I}$, then a sequence of operators $\{S_i\in \mathcal{L}(\mathcal{U};\mathcal{V}_i)\}_{i\in I}$ is called a g-frame of $\mathcal{U}$ with respect to $\{\mathcal{V}_i\}_{i\in I}$ if there exists positive constants $A,B$ such that the \textit{g-frame condition}
\begin{align}
    A\|u\|^2_{\mathcal{U}} \leq \sum_{i\in I} \|S_i u\|^2_{\mathcal{V}_i} \leq B\|u\|^2_{\mathcal{U}}
\end{align}
holds for all $u\in\mathcal{U}$. We call $\{S_i\}_{i\in I}$ a tight frame when $A=B$, and a Parseval frame when $A=B=1$. In our work we consider the case where $\mathcal{V}_i$ coincide for all $i$. When the g-frame condition holds, the g-frame operator
\begin{align*}
    \mathfrak{O}_S = \sum_{i\in I} S_i^* S_i
\end{align*}
is positive, bounded and invertible on $\mathcal{U}$. In \cite{skrett21}, g-frame operators of the type 
\begin{align*}
    \mathfrak{O}_S = \sum_{\lambda\in\Lambda} \pi(\lambda)S^*S\pi(\lambda)^*
\end{align*}
for some lattice $\Lambda$, were considered on the Hilbert space $L^2(\mathbb{R}^d)$. In this work, we say an operator $S\in\mathcal{L}(L^2(\mathbb{R}^{d}))$ generates a \textit{Gabor g-frame} if $\{S^*\pi(\lambda)^*\}_{\lambda\in\Lambda}$ is a frame for $\mathcal{HS}$. 
\begin{proposition}
If $S\in\mathcal{L}(L^2(\mathbb{R}^{d}))$ generates a Gabor g-frame of $\mathcal{HS}$ for some lattice $\Lambda$, then $S\in\mathcal{HS}$.
\end{proposition}
This follows from the same argument as \todo{M: should be capitalized?}Proposition 5.7 in \cite{skrett21}, taking a rank-one $T$.

For $S\in\mathcal{HS}$ which generates a Gabor g-frame, we define the analysis operator $C_S: \mathcal{HS} \to l^2(\Lambda;\mathcal{HS})$ by
\begin{align*}
    C_S T = \{S^*\pi(\lambda)^*T\}_{\lambda\in\Lambda},
\end{align*}
and the synthesis operator $D_S: l^2(\Lambda;\mathcal{HS}) \to \mathcal{HS}$ by
\begin{align*}
    D_S (\{T_{\lambda}\}_{\lambda\in\Lambda}) = \sum_{\lambda\in\Lambda} \pi(\lambda)S T_{\lambda}.
\end{align*}
If $S$ generates a Gabor g-frame, then general g-frame theory \cite{sun06} tells us there exists a \textit{canonical dual frame} 
\begin{align*}
    \{\Tilde{S}_{\lambda}\}_{\lambda\in\Lambda} := \{S^*\pi(\lambda)^*\mathfrak{O}_S^{-1}\}_{\lambda\in\Lambda}
\end{align*}
since $T=\mathfrak{O}_S^{-1}\mathfrak{O}_S T = \mathfrak{O}_S\mathfrak{O}_S^{-1} T$. $\Tilde{S}$ can be shown to be a Gabor g-frame generated by $(S^*\mathfrak{O}_S^{-1})$. We say in general that two operators $S,T\in\mathcal{L}(\mathcal{HS})$ generate \textit{dual Gabor g-frames} if $S$ and $T$ generate Gabor g-frames, and $\mathfrak{O}_{S,T} := D_S C_T = I_{\mathcal{HS}}$.

\subsection{Quantum Harmonic Analysis}
As a final prerequisite we present some theorems of quantum harmonic analysis, based on the convolutions introduced by Werner in \cite{werner84}, and recently applied to time-frequency anlysis in \cite{skrett18} \cite{skrett19} \cite{skrett20}, where it is used to generalise known results and provide more concise proofs by extending the mechanics of harmonic analysis to operators. We will on occasion use the framework of quantum harmonic analysis to simplify a proof or give an alternative framing. Convolutions between operators and functions are defined in the following manner;
\begin{definition}
For $f\in L^p(\mathbb{R}^{2d})$, $S\in\mathcal{S}^q$ and $T\in\mathcal{S}^p$, where $\frac{1}{p}+\frac{1}{q}=1+\frac{1}{r}$, convolutions are defined by
\begin{align*}
    f \star S &:= \int_{\mathbb{R}^{2d}} f(z) \alpha_z(S)\, dz \\
    S \star T &:= \mathrm{tr}(S\alpha_z(\Check{T}))
\end{align*}
where $\alpha_z(S)=\pi(z)S\pi(z)^*$ is a representation of the Weyl-Heisenberg group on $\mathcal{HS}$ and $\Check{T}=PTP$ where $P$ is the parity operator. The first integral is to be interpreted as a Bochner integral. 
\end{definition}
We will use a generalised version of Moyal's identity from \cite{werner84};
\begin{lemma}{(Generalised Moyal's Identity)}\label{generalmoyal}
For two operators $S,T\in\mathcal{S}^1$, the mapping $z\mapsto S\alpha_z(T)$ is integrable over $\mathbb{R}^{2d}$, and
\begin{align*}
    \int_{\mathbb{R}^{2d}} S\alpha_z(T)\, dz = \mathrm{tr}(S)\mathrm{tr}(T).
\end{align*}
\end{lemma}
Taking rank one operators returns precisely the original Moyal's identity, hence the name. We also note that the above holds when $T$ is replaced with $\Check{T}$, since tr$(T)=$tr$(PTP)$. We also make frequent use of the fact that for $S\in\mathcal{S}^1$;
\begin{align} \label{convidentity}
    1 \star S = \mathrm{tr}(S)I_{L^2},
\end{align}
which can be seen by using the spectral decomposition of $S$ and the reconstruction formula for $V_g$.

\section{An Operator STFT}
We start by defining the operator valued STFT. 
\begin{definition}{(Operator STFT)}
For two $\mathcal{HS}$ operators $S,T$ on $L^2(\mathbb{R}^d)$, the operator short-time Fourier transform, $\mathfrak{V}_S T$, is given by
\begin{equation}
    \mathfrak{V}_S T(z) = S^*\pi(z)^*T.
\end{equation}
\end{definition}
The operator STFT thus defines an operator valued function in phase space. We will see that this operator valued function is in many respects an analogue to the scalar function of the function STFT. To motivate such a definition, we consider the following:
\begin{example}
\normalfont For operators $S=\sum_n g_n \otimes e_n $ and $T=\sum_n f_n \otimes e_n$ with $f_n, g_n \in L^2(\mathbb{R}^d)$ and $\{e_n\}_n$ some orthonormal basis in $L^2(\mathbb{R}^d)$;
\begin{align*}
    \mathfrak{V}_S T(z) &= \sum_{n,m} V_{g_n}f_m(z) e_n \otimes e_m.
\end{align*}
\end{example}
\begin{remark}
\normalfont
    Here and in the sequel, we will often consider an operator $S=\sum_n f_n \otimes e_n$ where only the $e_n$ are assumed to be orthonormal. This is done because we will later consider different norms on the $f_n$. In the above example we could of course assume the $f_n$ to have $\|f_n\|_{L^2} = s_n$, where $s_n$ are the singular values of $S$.
\end{remark}
This definition is clearly equivalent to the definition in \cite{skrett22}, \cite{Guo22} in the case of a rank one $T=\psi \otimes \xi$, where we have $\mathfrak{V}_S T = (S^*\pi(z)^*\psi) \otimes \xi$, except that we consider the adjoint $S^*$. This adjustment is to make formulae in the sequel cleaner, and we note that there is no material difference in the two formulations. The STFT can thus be considered to encode information about time frequency correlations over functions. 
\begin{example}\label{locopex}
\normalfont Let $A_f^{\varphi_1,\varphi_2}$ and $A_g^{\psi_1,\psi_2}$ be standard single window localisation operators given by
\begin{align*}
    A_f^{\varphi_1,\varphi_2}&=\int_{\mathbb{R}^{2d}} f(z) \pi(z)\varphi_1\otimes\pi(z)\varphi_2\, dz \\
    A_g^{\psi_1,\psi_2}&=\int_{\mathbb{R}^{2d}} g(z) \pi(z)\psi_1\otimes\pi(z)\psi_2\, dz
\end{align*}
where $f,g\in L^2(\mathbb{R}^{2d})$ and $\varphi_i,\psi_i\in L^2(\mathbb{R}^d)$.
The operator STFT of $A_g^{\psi_1,\psi_2}$ with window $A_f^{\varphi_1,\varphi_2}$ is then
\begin{align*}
    \mathfrak{V}_{A_f^{\varphi_1,\varphi_2}} A_g^{\psi_1,\psi_2}(z) = \int_{\mathbb{R}^{4d}} \overline{f(z')}g(z'')\langle \pi(z'')\psi_1,\pi(z)\pi(z')\varphi_1\rangle \pi(z')\varphi_2\otimes \pi(z'')\psi_2 \,dz' \,dz''.
\end{align*}
This expression has an intuitive interpretation; if windows $\psi_1$ and $\varphi_1$ are concentrated in time-frequency around the origin, then the inner product in the integrand is negligible outside of the region around $z=z''-z'$. To illustrate this we consider the simple situation where $f,g$ are characteristic functions, and all windows are the Gaussian $\varphi_0$:
\begin{align*}
    \mathfrak{V}_{A_{\Omega_1}} A_{\Omega_2}(z) = \int_{\Omega_2} \int_{\Omega_1} \langle \pi(z'')\varphi_0,\pi(z)\pi(z')\varphi_0\rangle \pi(z')\varphi_0\otimes \pi(z'')\varphi_0 \,dz' \,dz'',
\end{align*}
where $A_{\Omega_i}:= A_{\chi_{\Omega_i}}^{\varphi_0,\varphi_0}$\todo{M: indices?}. \cref{fig:masks} shows some simple domains $\Omega_i$ in the time-frequency plane. In \cref{fig:spectrograms}, the resulting Hilbert-Schmidt norm of the operator STFTs $\mathfrak{V}_{A_{\Omega_i}} A_{\Omega_j}$ are shown as a function of $z$:

\begin{figure}[h]
\includegraphics[width=\textwidth]{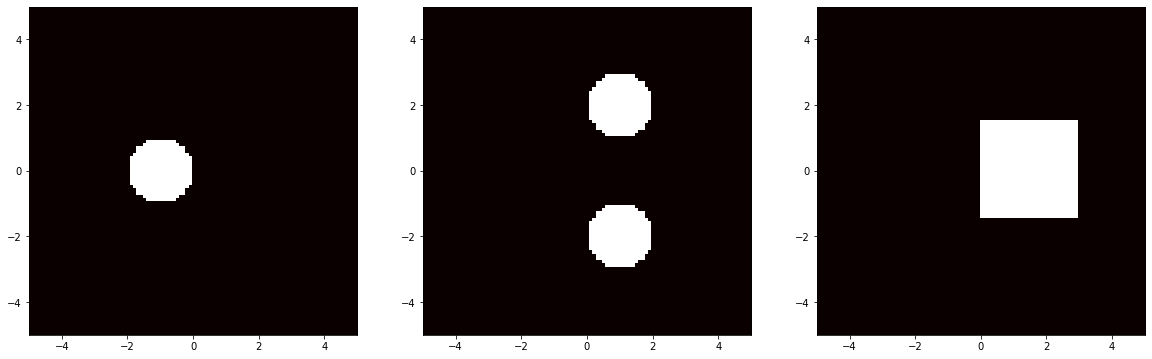}
\caption{From left to right: $\Omega_1$, $\Omega_2$, $\Omega_3$}
\label{fig:masks}
\end{figure}

\begin{figure}[h]
\includegraphics[width=1\textwidth]{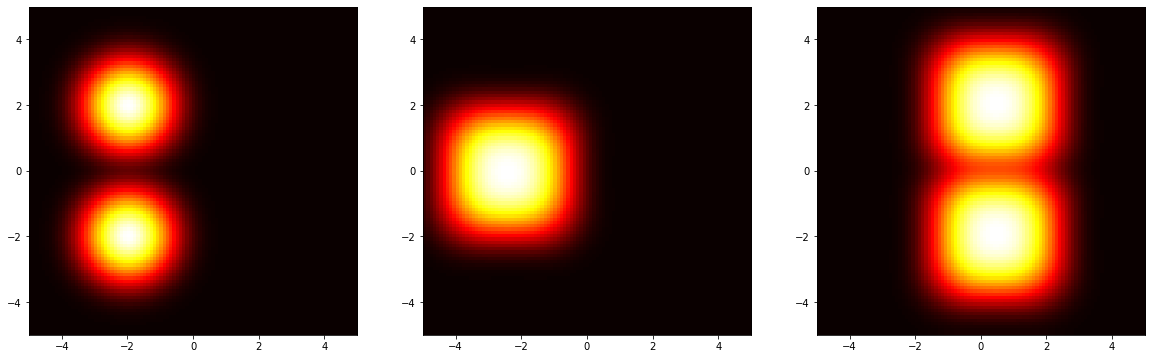}
\caption{From left to right: $\|\mathfrak{V}_{A_{\Omega_1}} A_{\Omega_2}(z)\|_{\mathcal{HS}}$, $\|\mathfrak{V}_{A_{\Omega_1}} A_{\Omega_3}(z)\|_{\mathcal{HS}}$, $\|\mathfrak{V}_{A_{\Omega_2}} A_{\Omega_3}(z)\|_{\mathcal{HS}}$}
\label{fig:spectrograms}
\end{figure}

These examples illustrates how the Hilbert-Schmidt norm of the operator STFT acts, but we also have the interpretation of $\pi(z')\varphi_0\otimes \pi(z'')\varphi_0$, as the operator sending time-frequency energy of a function from around $z''$ to around $z''-z$.

\end{example}

\begin{example}\label{dataopex}
\normalfont For a data operator $S=\sum_n f_n \otimes e_n$, 
\begin{align*}
    \mathfrak{V}_S S(z) = \sum_{n,m} V_{f_n} f_m(z) e_n\otimes e_m.
\end{align*}
Upon taking the taking the Hilbert-Schmidt norm, we recover the total correlation function from \cite{Dorf21};
\begin{align*}
    \|\mathfrak{V}_S S(z)\|_{\mathcal{HS}}^2 = \sum_{n,m} |V_{f_n} f_m(z)|^2.
\end{align*}
Hence the structure of the resulting operator can be seen to provide more information regarding the correlations within the dataset, as it relates where in the dataset the correlation occurs, for example on the diagonal versus off. To see this we compare two operators with identical total correlation functions, but one generated by functions determined by stationary process, while the other is generated by functions drawn from a non-stationary process. Both operators are of the form 
    \begin{align*}
        S_i = \sum_{n=1}^{200} f_i \otimes e_i
    \end{align*}
    where $e_i$ form an orthonormal basis, and $f_i$ are of the form
    \begin{align*}
        f_i = a_i \cdot sin(\mathrm{freq}_i\cdot t) \cdot g_i(t)
    \end{align*}
    where $a_i$ is a constant from a random normal distribution, $g_i$ a Bartlett-Hann window translated by a random $x$ from a normal distribution. The length of the signal is $200$, and so the resulting matrix $S_i$ has dimensions $200\times 200$. For the operator $S_1$, the frequencies $\mathrm{freq}_i$ are given by a base frequency with the addition of random noise from both a normal and sinoidal distribution. The operator $S_2$ on the other hand, is generated by the same base frequency, with random noise from a normal distribution, but with an $i$-dependent modulation. The two operators have an identical total correlation function, shown in \cref{fig:tc}, but comparing the operator-value of the STFT for different $z$'s shows the non-stationary structure of the function data set generating the operator $S_2$, \cref{fig:timedepstft}, when compared to $S_1$ \cref{fig:randstft}. In this respect the operator STFT can be seen to reflect the structure of an ordered data set, such as a functional time series.
    \begin{figure}
    \centering
        \includegraphics[width=.5\linewidth]{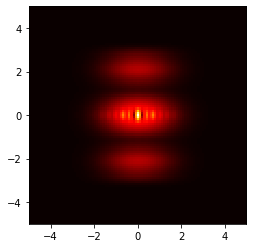}\hfill
        \caption{The common total correlation function of both $S_1$ and $S_2$}
        \label{fig:tc}
    \end{figure}
    \begin{figure}
        \includegraphics[width=.25\linewidth]{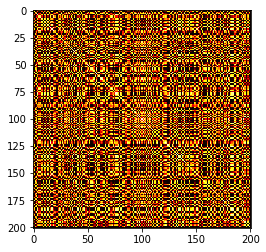}\hfill
        \includegraphics[width=.25\linewidth]{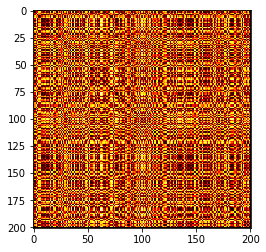}\hfill
        \includegraphics[width=.25\linewidth]{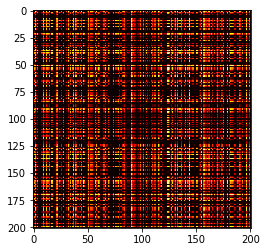}\hfill
        \includegraphics[width=.25\linewidth]{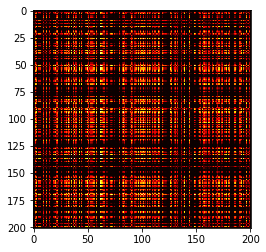}
      \caption{The matrix given by operator STFT $\mathfrak{V}_{S_1} S_1$ at $z=(0,0)$, $z=(0.5,0.5)$, $z=(1,1)$ and $z=(1.5,1.5)$.}
      \label{fig:randstft}
    \end{figure}
     \begin{figure}
        \includegraphics[width=.25\linewidth]{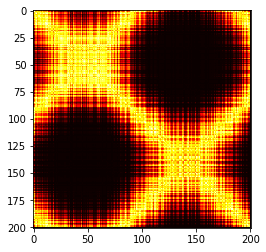}\hfill
        \includegraphics[width=.25\linewidth]{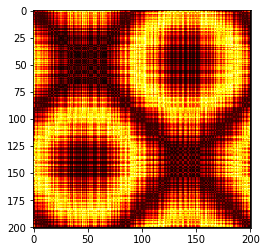}\hfill
        \includegraphics[width=.25\linewidth]{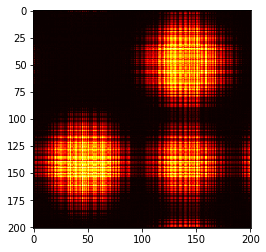}\hfill
        \includegraphics[width=.25\linewidth]{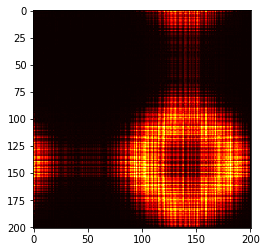}
      \caption{The matrix given by operator STFT $\mathfrak{V}_{S_2} S_2$ at $z=(0,0)$, $z=(0.5,0.5)$, $z=(1,1)$ and $z=(1.5,1.5)$.}
      \label{fig:timedepstft}
    \end{figure}
\end{example}
\todo{mention length of signal/size of matrix}

We collect some simple properties of the operator STFT:
\begin{proposition} \label{stftbasics}
    For operators $Q,R,S,T\in\mathcal{HS}$;
    \begin{enumerate}
        \item $\mathfrak{V}_S T(z) = e^{-2\pi i \omega x}(\mathfrak{V}_T S(-z))^*$
        \item $\int_{\mathbb{R}^{2d}} \langle \mathfrak{V}_S T(z), \mathfrak{V}_Q R(z)\rangle_{\mathcal{HS}}dz = \langle Q,S \rangle_{\mathcal{HS}} \langle T,R \rangle_{\mathcal{HS}}$
        \item $\int_{\mathbb{R}^{2d}} \| \mathfrak{V}_S T(z)\|^2_{\mathcal{HS}} = \|S\|^2_{\mathcal{HS}}\|T\|^2_{\mathcal{HS}}$
    \end{enumerate}
\end{proposition}
\begin{proof}
The first claim is merely a restatement of the property $\pi(z)^* = e^{-2\pi i x \omega}\pi(-z)$, and the third a special case of the second. To prove the second claim;
\begin{align*}
    \int_{\mathbb{R}^{2d}} \langle \mathfrak{V}_S T(z), \mathfrak{V}_Q R(z) \rangle_{\mathcal{HS}}\, dz &= \int_{\mathbb{R}^{2d}} \mathrm{tr}(S^*\pi(z)^*TR^*\pi(z)Q)\, dz \\
    &= \int_{\mathbb{R}^{2d}} \mathrm{tr}(TR^*\pi(z)QS^*\pi(z)^*)\, dz \\
    &= \int_{\mathbb{R}^{2d}} (TR^*)\star (PQS^*P) \, dz \\
    &= \langle Q,S \rangle_{\mathcal{HS}} \langle T,R \rangle_{\mathcal{HS}}
\end{align*}
where we have used \Cref{generalmoyal} in moving from the third to fourth line.

\end{proof}
In particular, the third statement gives us that $\mathfrak{V}_S:\mathcal{HS}\to L^2(\mathbb{R}^{2d};\mathcal{HS})$, and the mapping is continuous and injective. It is then natural to consider the Hilbert space adjoint, $\mathfrak{V}_S^*:L^2(\mathbb{R}^{2d};\mathcal{HS})\to\mathcal{HS}$, which is given by 
\begin{align*}
    \mathfrak{V}_S^* \Psi = \int_{\mathbb{R}^{2d}} \pi(z)S\Psi(z)\, dz
\end{align*}
for $\Psi(z)\in L^2(\mathbb{R}^{2d};\mathcal{HS})$. The integral can be interpreted in the weak sense in $\mathcal{HS}$. This can be seen directly;
\begin{align*}
    \langle \mathfrak{V}_S T, \Psi \rangle_{L^2(\mathbb{R}^{2d},\mathcal{HS})} &= \int_{\mathbb{R}^{2d}} \langle S^*\pi(z)^* T,\Psi(z)\rangle_{\mathcal{HS}}\, dz \\
    &= \int_{\mathbb{R}^{2d}} \mathrm{tr}(T\Psi(z)^*S^*\pi(z)^*)\, dz \\
    &= \int_{\mathbb{R}^{2d}} \langle T, \pi(z)S\Psi(z)\rangle_{\mathcal{HS}} \, dz.
\end{align*}
The operator STFT and its adjoint shares the reconstruction property with the function case, namely
\begin{align*}
    \mathfrak{V}_S^* \mathfrak{V}_R T &= \int_{\mathbb{R}^{2d}} \pi(z)SR^*\pi(z)^* T\, dz \\
    &= \int_{\mathbb{R}^{2d}} \alpha_z (SR^*) \, dz T\\
    &= (1\star SR^*) \cdot T = \langle S, R \rangle_{\mathcal{HS}} T
\end{align*}
where we use \eqref{convidentity}. We have as a result
\begin{align}\label{isometry}
    \mathfrak{V}_S^* \mathfrak{V}_S = I_{\mathcal{HS}}
\end{align}
for any $S\in\mathcal{HS}$ such that $\|S\|_{\mathcal{HS}}=1$. The converse then follows immediately, namely that for such an $S$,
\begin{align}\label{adjisometry}
    \mathfrak{V}_S \mathfrak{V}_S^*  = I_{\mathfrak{V}_S (L^2)}.
\end{align}

\section{Reproducing Kernel Structure}\label{repkernstruc}
In this section we examine the structure of the spaces generated by the operator STFT. We begin with the RKHS structure:
\begin{proposition}
    For any $S\in\mathcal{HS}$, the space
    \begin{align*}
        \mathfrak{V}_S(\mathcal{HS}) := \{\mathfrak{V}_S T : T\in\mathcal{HS} \}
    \end{align*}
    is a uniform reproducing kernel Hilbert space as a subspace of $L^2(\mathbb{R}^{2d};\mathcal{HS})$.
\end{proposition}
\begin{proof}
We start by confirming that the space is closed, since
\begin{align*}
    \|\mathfrak{V}_S T\|_{L^2(\mathbb{R}^{2d};\mathcal{HS})} = \int_{\mathbb{R}^{2d}}\|\mathfrak{V}_S T(z)\|_{\mathcal{HS}}^2 \, dz = \|S\|_{\mathcal{HS}}^2 \|T\|_{\mathcal{HS}}^2
\end{align*}
from \Cref{stftbasics}. Uniform boundedness of evaluation is quite straightforward;
\begin{align*}
    \|\mathfrak{V}_S T(z)\|_{\mathcal{HS}} = \|S^*\pi(z)^*T\|_{\mathcal{HS}} \leq \|S^*\pi(z)^*\|_{\mathcal{HS}}\|T\|_{\mathcal{HS}}
\end{align*}
\end{proof}
In fact we have already seen from \eqref{adjisometry} that the evaluation operator $E_z$ is given explicitly by $E_z = S^*\pi(z)^* \mathfrak{V}_S^*$, and so we must have that $E_z^* = \mathfrak{V}_S \pi(z)S$. By definition of the kernel function, we have for $S\in\mathcal{HS}$, with $\|S\|_{\mathcal{HS}}=1$, that
\begin{align*}
    K(z,z') &= E_z E_{z'}^* \\
    &= S^*\pi(z)^*\pi(z')S.
\end{align*}
\begin{remark}
\normalfont It should be noted that in the vector-valued RKHS setting the appearance of the operators $S$ and $\pi(z)$ (and their respective adjoints), in the definition of evaluation operator $E_z$ and its adjoint $E_z^*$, denotes the conjugation with these operators. As such we have that $E_z: \mathfrak{V}_S(\mathcal{HS})\to\mathcal{HS}$ and $E_z^*: \mathcal{HS} \to \mathfrak{V}_S(\mathcal{HS})$.
\end{remark}

This kernel is the integral kernel of the projection from $L^2(\mathbb{R}^{2d};\mathcal{HS})$ to $\mathfrak{V}_S(\mathcal{HS})$;
\begin{align} \label{rkhsproj}
    P_S \Psi(z) = \int_{\mathbb{R}^{2d}} K(z,z') \Psi(z') dz'
\end{align}
for $\Psi\in L^2(\mathbb{R}^{2d};\mathcal{HS})$. This defines\todo{formulation?} a projection, which can be seen from a simple calculation of $P_S^2$, and for any $T\in\mathcal{HS}$
\begin{align*}
    \int_{\mathbb{R}^{2d}} K(z,z') \mathfrak{V}_S T(z') dz' &= \mathfrak{V}_S\mathfrak{V}_S^*\mathfrak{V}_S T(z).
\end{align*}
Decomposing $S=\sum_n g_n \otimes e_n$, for orthonormal set $\{e_n\}_{n\in\mathbb{N}}$ and orthogonal set $\{g_n\}_{n\in\mathbb{N}}$ (where $g_i$ may be $0$), we find
\begin{align*}
    K(z,z') = \sum_{n,m\geq 0} \langle \pi(z')g_n, \pi(z)g_m\rangle_{L^2} \, e_m \otimes e_n.
\end{align*}
On the diagonals we have precisely the reproducing kernels of the scalar-valued Gabor spaces with windows $g_n$, that is to say kernels of the projections $V_{g_n} V_{g_n}^*$, but we have in addition the off-diagonal terms corresponding to the kernels of the maps $V_{g_n} V_{g_m}^*$. As a general property of RKHS', we have the inclusion
\begin{align*}
    \mathfrak{V}_S (\mathcal{HS}) \subset L^2(\mathbb{R}^{2d};\mathcal{HS}) \cap L^{\infty}(\mathbb{R}^{2d};\mathcal{HS}),
\end{align*}
since
\begin{align*}
    \|\mathfrak{V}_S (T)(z)\|^2_{\mathcal{HS}} &\leq \langle \mathfrak{V}_S (T), E_z^*E_z \mathfrak{V}_S (T)\rangle_{L^2(\mathbb{R}^{2d};\mathcal{HS})} \\
    &= \|\mathfrak{V}_S (T)\|^2_{L^2(\mathbb{R}^{2d};\mathcal{HS})}.
\end{align*}

\subsection{Characterisation from Twisted Convolution}
In an analogue way to the characterisation of Gabor space in terms of the twisted convolution, we can characterise the RKHS $\mathfrak{V}_S(\mathcal{HS})$ by the equivalent condition.
\begin{proposition}
Given $\Psi\in L^2(\mathbb{R}^{2d};\mathcal{HS})$, and $S\in\mathcal{HS}$ such that $\|S\|_{\mathcal{HS}}=1$;
\begin{align*}
    \Psi\natural\mathfrak{V}_S S = \Psi \iff \Psi \in \mathfrak{V}_S (\mathcal{HS}).
\end{align*}
\end{proposition}
\begin{proof}
On the one hand we have that for $Q,R,S,T\in\mathcal{HS}$;
\begin{align} \label{twistedconv}
    \mathfrak{V}_Q T \natural \mathfrak{V}_S R(z) &= \int_{\mathbb{R}^{2d}} S^*\pi(z-z')^*R Q^*\pi(z')^*Te^{-2\pi i x(\omega-\omega')}dz' \nonumber \\
    &= S^*\pi(z)^*\int_{\mathbb{R}^{2d}}\pi(z')R Q^*\pi(z')^*dz' T \nonumber \\
    &= \langle R, Q \rangle_{\mathcal{HS}} \mathfrak{V}_S T(z),
\end{align}
where the last inequality follows from \eqref{convidentity}, and hence the one direction follows in the case $Q=R=S$. On the other, from \eqref{rkhsproj}, 
\begin{align*}
    \Psi \natural \mathfrak{V}_S S(z) &= \int_{\mathbb{R}^{2d}} S^*\pi(z-z')^*S \Psi(z') e^{-2\pi i x(\omega-\omega')}\, dz' \\
    &= \int_{\mathbb{R}^{2d}} K(z,z') \Psi(z')\, dz' \\
    &= \big(P_S \Psi\big)(z)= \Psi(z)
\end{align*}
implies $\Psi\in \mathfrak{V}_S(\mathcal{HS})$.

\end{proof}

\subsection{Toeplitz operators}
With a RKHS structure, it is natural to consider what the corresponding Toeplitz operators on the space look like. Toeplitz operators are of the form $T_f = P_V M_f$, that is to say a pointwise multiplication by some $f\in L^{\infty}(\mathbb{R}^{2d})$, followed by a projection back onto the RKHS. In the case of Gabor spaces these are precisely the localisation or anti-Wick operators, which are accordingly also called Gabor-Toeplitz operators \cite{feich14}. Considering the Toeplitz operators on $\mathfrak{V}_S (\mathcal{HS})$, we have operators of the type
\begin{align*}
    T_f (\mathfrak{V}_S T) = \mathfrak{V}_S \mathfrak{V}_S^* (f\cdot\mathfrak{V}_S T)
\end{align*}
where $f\in L^{\infty}(\mathbb{R}^{2d})$ and $f\cdot\mathfrak{V}_S T$ is pointwise multiplication. We then define the unitarily equivalent operator $\Theta(T_f) := \mathfrak{V}_S^*T_f\mathfrak{V}_S$ on $\mathcal{HS}$: 
\begin{align*}
    \Theta (T_f)(T) &= \mathfrak{V}_S^*\mathfrak{V}_S \mathfrak{V}_S^* (f\cdot\mathfrak{V}_S T) \\
    &= \mathfrak{V}_S^* (f\cdot\mathfrak{V}_S T) \\
    &= \int_{\mathbb{R}^{2d}} f(z) \pi(z)SS^*\pi(z)^*T\, dz \\
    &= f \star (SS^*) T,
\end{align*}
and hence Toeplitz operators in the operator case correspond to the composition with the mixed-state localisation operators discussed in \cite{skrett19}.

\section{Coorbit Spaces for Operators}\label{opcoorbit}
From the previous section, we have a characterisation of the space $\mathfrak{V}_S(\mathcal{HS})$. We now turn to other classes which can be similarly characterised. In particular, from \Cref{stftbasics} the Hilbert-Schmidt operators are precisely the operators $\{T\in\mathcal{L}(L^2(\mathbb{R}^d)): \mathfrak{V}_S T\in L^2(\mathbb{R}^{2d};\mathcal{HS}) \}$ for $S \in \mathcal{HS}$, similarly to the function case of $L^2(\mathbb{R}^d)=M^2(\mathbb{R}^d)$. We therefore set out to define what we refer to as \textit{operator coorbit spaces}. In the sequel, $v(z)$ will be a sub-multiplicative weight function of polynomial growth on phase space, and $m(z)$ will be a $v$-moderate weight function on phase space.
\subsection{The $\mathfrak{M}^1_v$ case}
In a similar vein to the function case we define the admissible operators, for a weight function $v$, to be 
\begin{align}\label{admcond}
    \mathfrak{A}_v := \{S\in\mathcal{HS}: \mathfrak{V}_S S \in L^1_v(\mathbb{R}^{2d};\mathcal{HS}) \}.
\end{align}
\begin{example}
\normalfont Clearly any rank one operator which can be written as $T=f\otimes \psi$, where $f\in M^1_v(\mathbb{R}^d)$ and $\psi\in L^2(\mathbb{R}^d)$, is in $\mathfrak{A}_v$.
\end{example}

We set $S_0 = \varphi_0 \otimes \varphi_0$, and define the space
\begin{align}\label{m1cond}
    \mathfrak{M}^1_v := \{T \in \mathcal{HS}: \mathfrak{V}_{S_0} T \in L^1_v(\mathbb{R}^{2d};\mathcal{HS}) \}
\end{align}
with corresponding norm $\|T\|_{\mathfrak{M}^1_v} = \|\mathfrak{V}_{S_0} T\|_{L^1_v(\mathbb{R}^{2d};\mathcal{HS})}$, and we denote the unweighted version $v(z)\equiv 1$ by $\mathfrak{M}^1$. 
\begin{remark}\label{tfintuition}
\normalfont Considering $\|(\varphi_0\otimes\varphi_0)\pi(z)^*T\|_{\mathcal{HS}}$, it is easy to see how the $\mathfrak{M}^{1}_v$ condition (and later the $\mathfrak{M}^{p,q}_m$ conditions) can be seen to measure the time-frequency localisation of an operator. In this case, the $\mathfrak{M}^1_v$ condition is simply a measure of how time-frequency translations of $\varphi_0$ decay as arguments of $T^*$: $\int v(z)\|T^*(\pi(z)\varphi_0)\|_{L^2} dz$. Following this line of reasoning, we consider appropriate localisation operators of the type in \cref{locopex}:
\end{remark}
\todo{Say what happens next}
\begin{example}
\normalfont For a localisation operator $A_h^{\psi}$, if $h\in L^1(\mathbb{R}^{2d})$ and $\psi\in M^1(\mathbb{R}^d)$, then $A_h^{\psi}\in\mathfrak{M}^1$;\todo{something is unclear in first step- what is $A_{\overline{h}}^{\psi}(z)$? }
\begin{align*}
    \int_{\mathbb{R}^{2d}} \|\mathfrak{V}_{S_0} A_h^{\psi}(z)\|_{\mathcal{HS}}dz &= \int_{\mathbb{R}^{2d}} \| \varphi_0\otimes \varphi_0 \pi(z)^* A_{h}^{\psi}\|_{\mathcal{HS}}dz \\
    &= \int_{\mathbb{R}^{2d}} \| A_{\overline{h}}^{\psi}\pi(z)\varphi_0\|_{L^2}dz \\
    &= \int_{\mathbb{R}^{2d}} \| \int_{\mathbb{R}^{2d}} \overline{h}(z')\langle \pi(z)\varphi_0,\pi(z')\psi\rangle \pi(z')\psi\, dz' \|_{L^2} dz \\
    &\leq \int_{\mathbb{R}^{2d}} |\overline{h}(z')| \int_{\mathbb{R}^{2d}} |V_{\psi} \varphi_0 (z-z')|\, dz\, dz'.
\end{align*}
\end{example}

Since $\mathfrak{V}_{S_0}(\mathcal{HS})$ is a RKHS, it is clear that $\mathfrak{M}^1_v \subset \mathcal{HS}$. This inclusion is continuous, since
\begin{align*}
    \|T\|_{\mathcal{HS}}^2 &= \|\mathfrak{V}_{S_0} T\|_{L^2(\mathbb{R}^{2d};\mathcal{HS})}^2 \\
    &= \int_{\mathbb{R}^{2d}} v(z)\|\mathfrak{V}_{S_0} T(z) \|_{\mathcal{HS}}\|\mathfrak{V}_{S_0} T(z) \|_{\mathcal{HS}}\, dz \\
    &\leq \|T\|_{\mathcal{HS}} \|T\|_{\mathfrak{M}^1_v}
\end{align*}
where we have used that $\|\mathfrak{V}_S T(z)\|_{\mathcal{HS}}\leq\|S\|_{\mathcal{HS}}\|T\|_{\mathcal{HS}}$ for every $z$. We can hence decompose every $T \in \mathfrak{M}^1_v$ as $T=\sum_{n\geq 0} f_n \otimes e_n$ for some orthonormal system $\{e_n\}_n$ and orthogonal $\{f_n\}_n$. The $\mathfrak{M}^1_v$ condition \Cref{m1cond} is then equivalent to 
\begin{align*}
    \mathfrak{M}^1_v = \{ T = \sum_n f_n \otimes e_n \in \mathcal{HS}: \int_{\mathbb{R}^{2d}} v(z)\|V_{\varphi_0} f_n (z)\|_{l^2(\mathbb{N})} dz<\infty \}.
\end{align*}
Noting that
\begin{align*}
    \|T\|_{\mathfrak{M}^1_v} \geq \int_{\mathbb{R}^{2d}} v(z) |V_{\varphi_0} f_n (z)| dz
\end{align*}
for each $n$, we find that $f_n \in M^1_v(\mathbb{R}^d)$ for all $n$ when $T\in \mathfrak{M}^1_v$, with $\|f\|_{M^1_v} \leq \|T\|_{\mathfrak{M}^1_v}$. 
\begin{claim} \label{nuclearinc}
The space of nuclear operators $\mathcal{N}(L^2(\mathbb{R}^d);M^1_v(\mathbb{R}^d))$ is contained in $\mathfrak{M}^1_v$.
\end{claim}
\begin{proof}
\normalfont Taking some $T\in \mathcal{N}(L^2(\mathbb{R}^d);M^1_v(\mathbb{R}^d))$, we can decompose $T=\sum_n f_n\otimes g_n$, with $\sum_n \|f_n\|_{M^1_v}\|g_n\|_{L^2} < \infty$, we assume without loss of generality that $\|g\|_{L^2}=1$ for all $n$. Here neither the $f_n$ or $g_n$ are necessarily orthogonal. We have that
\begin{align*}
    \mathfrak{V}_{S_0} T(z) =\sum_{n} V_{\varphi_0} f_n(z) \varphi_0 \otimes g_m.
\end{align*}
It then follows that
\begin{align*}
    \|\mathfrak{V}_{S_0} T(z) \|_{\mathcal{HS}} &= \big(\sum_{n,m} |\langle g_m, g_n \rangle V_{\varphi_0} f_n(z)V_{\varphi_0} f_m(z)|\big)^{1/2}  \\
    &\leq \big(\sum_{n,m} |V_{\varphi_0} f_n(z)V_{\varphi_0} f_m(z)|\big)^{1/2} \\
    &= \sum_{n} |V_{\varphi_0} f_n(z)|,
\end{align*}
since we assumed $\|g_n\|_{L^2}=1$. The nuclear condition thus gives that
\begin{align*}
    \int_{\mathbb{R}^{2d}} v(z) \|\mathfrak{V}_{S_0} T(z) \|_{\mathcal{HS}} \, dz &\leq \int_{\mathbb{R}^{2d}} \sum_n v(z) |V_{\varphi_0} f_n(z)|\, dz \\
    &= \sum_n \int_{\mathbb{R}^{2d}} v(z) |V_{\varphi_0} f_n(z)|\, dz \\
    &= \sum_n \|f_n\|_{M^1_v} \\
    &< \infty.
\end{align*}
We conclude that $\mathcal{N}(L^2(\mathbb{R}^d);M^1_v(\mathbb{R}^d)) \subset \mathfrak{M}^1_v$
\end{proof}

\begin{remark} \label{nuclear2inc}
\normalfont Any $T\in \mathfrak{M}^1_v$ can be written in the form  $\sum_n f_n \otimes e_n$, where $\|e_n\|_{L^2}=1$ for all $n$, and $\{\|f_n\|_{M^1_v}\}_n\in l^2$. This follows from the inequality 
\begin{align*}
    \big\| \int_{\mathbb{R}^{2d}} |V_{\varphi_0} f_n(z)| \, dz \big\|_{l^2} \leq \int_{\mathbb{R}^{2d}} \|V_{\varphi_0} f_n(z)\|_{l^2}\, dz.
\end{align*}
As a result, for the unweighted case, $S\in\mathfrak{M}^1 \implies SS^*\in \mathcal{N}(M^1(\mathbb{R}^2);M^1(\mathbb{R}^2))$, or alternatively $\sigma_{SS^*}\in M^1(\mathbb{R}^{2d})$ where $\sigma_{SS^*}$ is the Weyl symbol of $SS^*$ \cite{jakob22}.
\end{remark}

\begin{remark} \label{rankoneschatten}
\normalfont
An operator $T=\sum_n f_n\otimes e_n$ in the space $\mathfrak{M}^1_v$ also satisfies the condition
\begin{align*}
    \mathfrak{V}_{S_0} T \in L^1(\mathbb{R}^{2d};\mathcal{S}^p),
\end{align*}
since $\mathfrak{V}_{S_0} T$ takes the values of rank one operators, and so all Schatten class norms coincide. However, as we will later see, using the Hilbert-Schmidt norm is required when considering results for a wider class of window operators. Hence for the sake of consistency we define the operator coorbit spaces in terms of the Hilbert-Schmidt norm.
\end{remark}

As a corollary of \Cref{nuclearinc}, operators $T\in M^1_v(\mathbb{R}^d) \hat{\otimes}_{\pi} M^1_v(\mathbb{R}^d)$, and in the case of polynomial growth of $v$ the Schwartz operators $\mathfrak{S}$, are contained in $\mathfrak{M}^1_v$. We will use this to give a suitably large reservoir for defining general coorbit spaces. 

\subsection{The general $\mathfrak{M}^{p,q}_m$ case}
We then define the operator coorbit spaces for $1\leq p,q \leq \infty$ and $v$-moderate weight $m$ by
\begin{align*}
    \mathfrak{M}^{p,q}_m := \{T \in \mathfrak{S}': S_0^* \pi(z)^* T \in L^{p,q}_m(\mathbb{R}^{2d};\mathcal{HS}) \}.
\end{align*}
with norms $\|T\|_{\mathfrak{M}^{p,q}_m} = \|\mathfrak{V}_{S_0} T\|_{L^{p,q}_m}$. 
\begin{example}
\normalfont As in the $\mathfrak{M}^1_v$ case, any rank one operator which can be written as $T=f\otimes \psi$, where $f\in M^{p,q}_v(\mathbb{R}^d)$ and $\psi\in L^2$, is in $\mathfrak{M}^{p,q}_v$.
\end{example}
\begin{remark}
\normalfont Since we restrict our focus to weights of polynomial growth, the Schwartz operator dual is a sufficiently large reservoir, although if we wished to extend to a larger class of weights this may fail. For weights of exponential growth one has to use ultradistributions \cite{pite04,cato17} and it seems to be a promising topic for future research to study these kind of objects in our setting, too. 
\end{remark}
We use the notation $\mathfrak{V}_S T(z) = S^*\pi(z)^* T$ for $S\in\mathfrak{M}^1_v$ and $T\in \mathfrak{M}^{p,q}_m$, and similarly $\mathfrak{V}_S^* \Psi = \int_{\mathbb{R}^{2d}} \pi(z)S\Psi(z)dz$ for $\Psi\in L^{p,q}_m(\mathbb{R}^{2d};\mathcal{HS})$. The map $\mathfrak{V}_S$ is injective, as for any non-zero $R\in\mathfrak{M}^{p,q}_m$ there exists some $f\in L^2(\mathbb{R}^d)$ such that $Rf$ is non-zero, and so the injectivity of $\mathfrak{V}_S$ follows from the properties of the function STFT. The $\mathfrak{M}^{p,q}_m$ spaces are clearly closed under addition and scalar multiplication. To show that they are in fact Banach spaces, we use the following lemma:
\begin{lemma}\label{projidentity}
For $1\leq p \leq \infty$ and $S\in\mathfrak{A}_v$, the map $\mathfrak{V}_S \mathfrak{V}_S^*$ is a bounded operator on $L^{p,q}_m(\mathbb{R}^{2d};\mathcal{HS})$, and if $\|S\|_{\mathcal{HS}}=1$ then its restriction to $\mathfrak{V}_S (\mathfrak{M}^{p,q}_m)$ is the identity.
\end{lemma}
\begin{proof}
We begin by noting that for $\Psi \in L^{p,q}_m(\mathbb{R}^{2d};\mathcal{HS})$;
\begin{align*}
    \mathfrak{V}_S \mathfrak{V}_S^*\Psi(z) &= \int_{\mathbb{R}^{2d}} K(z,z') \Psi(z') dz' \\
    &= \int_{\mathbb{R}^{2d}} S^*\pi(z)^*\pi(z')S \Psi(z') dz' \\
    &= \Psi\natural \mathfrak{V}_S S(z).
\end{align*}
Hence from \Cref{youngsineq};
\begin{align*}
    \| \mathfrak{V}_S \mathfrak{V}_S^*\Psi \|_{L^{p,q}_m} &= \|\Psi\natural \mathfrak{V}_S S(z)\|_{L^{p,q}_m} \\
    &\leq C_{m,v}\|\Psi\|_{L^{p,q}_m} \|\mathfrak{V}_S S\|_{L^1_v}
\end{align*}
and so $\mathfrak{V}_S \mathfrak{V}_S^*$ is bounded, since $S\in\mathfrak{A}_v$. For $\mathfrak{V}_S T \in L^{p,q}_m(\mathbb{R}^{2d};\mathcal{HS})$, as in the $\mathcal{HS}$ case we observe 
\begin{align*}
    \mathfrak{V}_S \mathfrak{V}_S^* \mathfrak{V}_S T &= \mathfrak{V}_S \int_{\mathbb{R}^{2d}} \pi(z)S S^*\pi(z)^*T\, dz \\
    &= \mathfrak{V}_S \int_{\mathbb{R}^{2d}} \alpha_z(S^*S)\, dz\, T = \mathfrak{V}_S T.
\end{align*}
\end{proof}
\begin{corollary}\label{stftinvmod}
    For $1\leq p \leq \infty$ and $S\in\mathfrak{A}_v$, $\mathfrak{V}_S^*$ is a bounded map from $L^{p,q}_m(\mathbb{R}^{2d};\mathcal{HS})$ to $\mathfrak{M}^{p,q}_m$.
\end{corollary}

\begin{proposition}
    For $1\leq p \leq \infty$, $1\leq q \leq \infty$, and $v$-moderate $m$, $\mathfrak{M}^{p,q}_m$ is a Banach space.
\end{proposition}
\begin{proof}
We consider a Cauchy sequence $\{T_n\}_{n\in\mathbb{N}} \subset \mathfrak{M}^{p,q}_m$. The sequence $\{\mathfrak{V}_{S_0} T_n\}_{n\in\mathbb{N}}$ is a Cauchy sequence in $L^{p,q}_m(\mathbb{R}^{2d})$ by definition of the norm, and since $L^{p,q}_m(\mathbb{R}^{2d};\mathcal{HS})$ is a Banach space we denote the limit of this sequence $\Psi$. From \cref{stftinvmod} $\mathfrak{V}_{S_0}^* \Psi\in \mathfrak{M}^{p,q}_m$ and $T_n \to \mathfrak{V}_{S_0}^* \Psi$ by boundedness, so $\mathfrak{M}^{p,q}_m$ are Banach spaces. 

\end{proof}
As in the function case we have the embedding of our spaces:
\begin{claim}
For $1\leq p \leq p' \leq \infty$, $1\leq q \leq q' \leq \infty$, and $m(z)\geq m'(z)$ , $\mathfrak{M}^{p,q}_m \subset \mathfrak{M}^{p',q'}_{m'}$.
\end{claim}
This follows from the reproducing formula for $\mathfrak{M}^{p,q}_m$ and the previous lemma; $\mathfrak{V}_{S_0} \mathfrak{V}_{S_0}^* \mathfrak{V}_{S_0} T (z) = S_0^*\pi(z)^*\mathfrak{V}_{S_0}^* \mathfrak{V}_{S_0} T$. Hence $\mathfrak{V}_{S_0}\mathfrak{M}^{p,q}_m \subset L^{p,q}_{m'}(\mathbb{R}^{2d};\mathcal{HS}) \cap L^{\infty}_{m'}(\mathbb{R}^{2d};\mathcal{HS})$ and the claim follows.

\subsection{Equivalent Norms} \label{equivnorms}
The twisted convolution structure can be used to show that as in the function case, different operators in $\mathfrak{M}^1_v$ generate the same $\mathfrak{M}^{p,q}_m$ spaces, with equivalent norms.
\begin{proposition}\label{equivnorms}
Given some $R \in \mathfrak{M}^1_v$, the space $\{T\in\mathcal{HS}: \mathfrak{V}_R T\in L^{p,q}_m(\mathbb{R}^{2d};\mathcal{HS})\}$ is equal to the space $\mathfrak{M}^{p,q}_m$, and the associated norms are equivalent.
\end{proposition}
\begin{proof}
 Given $R\in\mathfrak{M}^1_v$, and $T\in\mathfrak{M}^{p,q}_m$, we aim to show that $\mathfrak{V}_R T\in L^{p,q}_m(\mathbb{R}^{2d};\mathcal{HS})$. To that end we have from \Cref{youngsineq} that
\begin{align*}
    \|\mathfrak{V}_R T\|_{L^{p,q}_m} &= \frac{1}{\|S_0\|^2_{\mathcal{HS}}} \|\mathfrak{V}_{S_0} T \natural \mathfrak{V}_R S_0\|_{L^{p,q}_m} \\
    &\leq C_{v,m}\|\mathfrak{V}_{S_0} T\|_{L^{p,q}_m}  \|\mathfrak{V}_{S_0} R\|_{L^1_v} \\
    &< \infty.
\end{align*}
where $C_{v,m}$ is the $v$-moderate constant of $m$, and we have used \Cref{stftbasics}(i). We have also used the formula $\mathfrak{V}_Q T \natural \mathfrak{V}_S R(z) = \langle R, Q \rangle_{\mathcal{HS}} \mathfrak{V}_S T(z)$, which we initially defined only for $T\in\mathcal{HS}$. However examining the argument confirms we are justified in using this for general $T$. Hence $\mathfrak{M}^{p,q}_m \subset \{T\in\mathcal{HS}: \mathfrak{V}_R T\in L^{p,q}_m(\mathbb{R}^{2d};\mathcal{HS})\}$. Conversely, repeating the above argument with $T$ such that $\mathfrak{V}_R T\in L^{p,q}_m(\mathbb{R}^d)$ gives the reverse inclusion. Equivalence of norms is also clear from these symmetric arguments, namely
\begin{align*}
    \frac{\|R\|_{\mathcal{HS}}^2}{C_{v,m}\|R\|_{\mathfrak{M}^1_v}} \|T\|_{\mathfrak{M}^{p,q}_m} \leq \|\mathfrak{V}_R T\|_{L^{p,q}_m}  \leq \frac{C_{v,m}\|R\|_{\mathfrak{M}^1_v}}{\|S_0\|_{\mathcal{HS}}^2} \|T\|_{\mathfrak{M}^{p,q}_m}.
\end{align*}
\end{proof}

\begin{corollary}\label{admissableops}
    $\mathfrak{M}^1_v = \mathfrak{A}_v$. 
\end{corollary}
\begin{proof}
It is clear that $\mathfrak{M}^1_v \subset \mathfrak{A}_v$. The inclusion $\mathfrak{A}_v \subset \mathfrak{M}^1_v$ follows from a similar argument as above. Given some $T\in\mathfrak{A}_v$ such that $\langle S, T\rangle_{\mathcal{HS}} \neq 0$;
\begin{align*}
    \|\mathfrak{V}_{S_0} T\|_{L^1_v} \leq \frac{1}{\langle S_0, T\rangle_{\mathcal{HS}}}\|\mathfrak{V}_T T\|_{L^1_v} \|\mathfrak{V}_{S_0} S_0 \|_{L^1_v}.
\end{align*}
In the case that $\langle S, T\rangle_{\mathcal{HS}} = 0$, we simply take some $R\in\mathfrak{M}^1_v$ with $\langle S, R\rangle_{\mathcal{HS}} \neq 0$ and $\langle T, R\rangle_{\mathcal{HS}} \neq 0$ and repeat the above expansion twice with respect to $R$ to derive
\begin{align*}
    \|\mathfrak{V}_{S_0} T\|_{L^1_v} \leq \frac{1}{\langle T, S_0\rangle_{\mathcal{HS}}\langle R, T\rangle_{\mathcal{HS}}}\|\mathfrak{V}_T T\|_{L^1_v} \|\mathfrak{V}_{S_0} S_0 \|_{L^1_v}\|\mathfrak{V}_{R} R \|_{L^1_v}.
\end{align*}

\end{proof}

\begin{corollary}
    A Hilbert-Schmidt operator $S$ belongs to the space $\mathfrak{A}_v$ if and only if the following norm equivalence holds: 
    \begin{align}\label{equivnormfunc}
        \|S^*\pi(z)^*f\|_{L^{p,q}_m(\mathbb{R}^{2d};L^2)} \asymp \|f\|_{M^{p,q}_m}
    \end{align}
    For every $1\leq p,q \leq \infty$, $v$-multiplicative $m$.
\end{corollary}
\begin{proof}
The $M^{p,q}_m(\mathbb{R}^d)$ condition $\|V_{\varphi_0}f\|_{L^{p,q}_m(\mathbb{R}^{2d})} < \infty$ is equivalent to the $\mathfrak{M}^{p,q}_m$ condition $\|\mathfrak{V}_{S_0} (f \otimes \varphi_0)\|_{L^{p,q}_m(\mathbb{R}^{2d};\mathcal{HS})} < \infty$. From \Cref{equivnorms}, all $S\in\mathfrak{A}_v$ determine equivalent norms on these spaces. Conversely for any operator $S$ satisfying \Cref{equivnormfunc}, for all $1\leq p,q \leq \infty$ and all $v$-multiplicative $m$, satisfies $\|\mathfrak{V}_S (f \otimes \varphi_0)\|_{L^{p,q}_m(\mathbb{R}^{2d};\mathcal{HS})}<\infty$, and in particular $\|\mathfrak{V}_S S_0\|_{L^1_v(\mathbb{R}^{2d};\mathcal{HS})} < \infty$, so $S$ must be in $\mathfrak{A}_v$ by \cref{admissableops}.
\end{proof}

\subsection{Duality}
To show the duality property $(\mathfrak{M}^{p,q}_m)' \cong \mathfrak{M}^{p',q'}_{1/m}$ where $\frac{1}{p} + \frac{1}{p'} = 1$ and similarly for $q$, we follow a similar approach to the function case proof in \cite{grochenigtfa}. We will however need a result of \cite{Gret72} for Lebesgue-Bochner spaces:
\begin{lemma}\label{bochnerdual}
For a Banach space $B$, and $\sigma$-finite measure space $(\Omega,\mathcal{A},\mu)$, $B$ has the Radon-Nikodym property (RNP) if and only if 
\begin{align*}
    L^p(\Omega;B)' = L^q(\Omega;B')
\end{align*}
with dual action 
\begin{align*}
    \langle a, a^* \rangle_{B,B'} = \int_{\Omega} a^*(a)\, d\mu
\end{align*}
where $\frac{1}{p}+\frac{1}{q} = 1$ for $1\leq p < \infty$.
\end{lemma}
Since $\mathcal{HS}$ has the RNP this gives that $\big(L^{p,q}_{m}(\mathbb{R}^{2d};\mathcal{HS})\big)' \cong L^{p',q'}_{1/m}(\mathbb{R}^{2d};\mathcal{HS})$, with the dual action $\langle \Psi, \Phi\rangle_{L^{p,q}_{m},L^{p',q'}_{1/m}} = \int_{\mathbb{R}^{2d}} \langle \Psi(z), \Phi(z)\rangle_{\mathcal{HS}}\, dz$.
\begin{proposition} \label{dualspaces}
For $S\in \mathfrak{A}_v$ and $1\leq p < \infty$, we have the duality identity
\begin{align*}
    (\mathfrak{M}^{p,q}_m)' \cong \mathfrak{M}^{p',q'}_{1/m}
\end{align*}
with the dual action given by
\begin{align*}
    \langle T, R \rangle_{\mathfrak{M}^{p,q}_m,\mathfrak{M}^{p',q'}_{1/m}} = \int_{\mathbb{R}^{2d}} \langle \mathfrak{V}_S T(z), \mathfrak{V}_S R(z)\rangle_{\mathcal{HS}}\, dz.
\end{align*}
\end{proposition}
\begin{proof}
On the one hand, the inclusion $\mathfrak{M}^{p',q'}_{1/m} \subset (\mathfrak{M}^{p,q}_m)'$ is clear from Hölder's inequality for weighted mixed norm spaces;
\begin{align*}
    \Big| \int_{\mathbb{R}^{2d}} \langle \mathfrak{V}_S T(z), \mathfrak{V}_S R(z)\rangle_{\mathcal{HS}}\, dz \Big| \leq \|T\|_{\mathfrak{M}^{p,q}_m}\|R\|_{\mathfrak{M}^{p',q'}_{1/m}}.
\end{align*}
To demonstrate the converse, take $R\in (\mathfrak{M}^{p,q}_m)'$. The composition $\Tilde{R} := R\circ \mathfrak{V}_S^*$ then defines a functional on $L^{p,q}_m(\mathbb{R}^{2d};\mathcal{HS})$, by \Cref{stftinvmod}. There exists then some $\Theta(z) \in L^{p',q'}_{1/m}(\mathbb{R}^{2d};\mathcal{HS})$, due to \Cref{bochnerdual}, such that 
\begin{align*}
    \Tilde{R}(\Psi) = \int_{\mathbb{R}^{2d}} \langle \Psi(z), \Theta(z)\rangle_{\mathcal{HS}}\, dz
\end{align*}
for $\Psi\in L^{p,q}_m((\mathbb{R}^{2d};\mathcal{HS})$. From \Cref{stftinvmod} it follows that
\begin{align*}
    \mathfrak{V}_S^*\Theta=\int_{\mathbb{R}^{2d}}\pi(z) S \Theta(z)\, dz \in \mathfrak{M}^{p',q'}_{1/m}
\end{align*}
and we denote this element $\theta$. We then conclude by confirming that 
\begin{align*}
    \langle T, \theta \rangle_{\mathfrak{M}^{p,q}_m,\mathfrak{M}^{p',q'}_{1/m}} &= \int_{\mathbb{R}^{2d}} \langle \mathfrak{V}_S T(z), \mathfrak{V}_S\mathfrak{V}_S^*\Theta(z) \rangle_{\mathcal{HS}} dz \\
    &= \int_{\mathbb{R}^{2d}} \langle \mathfrak{V}_S T(z), \Theta(z) \rangle_{\mathcal{HS}} dz \\
    &= \Tilde{R}(\mathfrak{V}_S T) = R(T),
\end{align*}
i.e. that an arbitrary functional $R \in (\mathfrak{M}^{p,q}_m)'$ corresponds to an element $\theta\in\mathfrak{M}^{p',q'}_{1/m}$ with the dual action defined above. 
Thus we have shown the reverse inclusion of $(\mathfrak{M}^{p,q}_m)^* \subset \mathfrak{M}^{p',q'}_{1/m}$, and conclude $(\mathfrak{M}^{p,q}_m)' \cong \mathfrak{M}^{p',q'}_{1/m}$.
\end{proof}

\begin{remark}\label{Schatten behaviour}
\normalfont In examples so far of $\mathfrak{M}^{p,q}_m$ operators, we have considered the rank one case, where one retrieves the familiar functions in the $M^{p,q}_m$ spaces and their associated relations. However, the $\mathfrak{M}^{p,q}_m(\mathbb{R}^d)$ spaces can also be related to the Schatten properties of operators, as seen by the inclusions
\begin{align*}
    \mathcal{N}(L^2;M^1)\subseteq \mathfrak{M}^1 \subset \mathcal{S}^1 \subset \mathcal{HS} \subset \mathcal{L}(L^2) \subset \mathfrak{M}^{\infty} \subseteq \mathcal{L}(L^2;M^{\infty}).
\end{align*}

\end{remark}

In particular we have a Gelfand triple $\mathfrak{M}^1_m \subset \mathcal{HS} \subset \mathfrak{M}^{\infty}_{1/m}$, where the embeddings are continuous.

\subsection{Correspondence Principle for Operators}
Finally we can give a characterisation of the spaces in terms of a coorbit structure: 
\begin{theorem}
    For any $S\in \mathfrak{A}_v$ such that $\|S\|_{\mathcal{HS}}=1$, we have an isometric isomorphism
    \begin{align*}
        \mathfrak{M}^{p,q}_m := \{ T\in \mathfrak{S}': \mathfrak{V}_S T \in L^{p,q}_m(\mathbb{R}^{2d};\mathcal{HS}) \} \cong \{\Psi\in L^{p,q}_m(\mathbb{R}^{2d};\mathcal{HS}): \Psi = \Psi \natural \mathfrak{V}_S S\},
    \end{align*}
    under the mapping
    \begin{align*}
        T \mapsto \mathfrak{V}_S T.
    \end{align*}\todo{$\Psi$?}
\end{theorem}
\begin{proof}
The inclusion $\mathfrak{V}_S(\mathfrak{M}_S^p) \subset \{\Psi\in L^p(\mathbb{R}^{2d};\mathcal{HS}): \Psi = \Psi \natural \mathfrak{V}_S S\}$ follows from \Cref{projidentity}. It remains to show the converse. We have that $\Psi \natural \mathfrak{V}_S S = \mathfrak{V}_S\mathfrak{V}_S^* \Psi$ for any $\Psi\in L^{p,q}_m(\mathbb{R}^{2d};\mathcal{HS})$. Hence if $\Psi = \Psi \natural \mathfrak{V}_S S$, then $\Psi = \mathfrak{V}_S R$, where $R=\mathfrak{V}_S^* \Psi \in\mathfrak{M}^{p,q}_m$, since $\mathfrak{V}_S^*:L^{p,q}_m\to\mathfrak{M}^{p,q}_m$. We recall that $\mathfrak{V}_S$ is injective on $\mathfrak{M}^{p,q}_m$, and the isometry property follows simply as a result of definitions of $\mathfrak{M}^{p,q}_m$ norms for a normalised $S$. Hence we have the correspondence principle;
\begin{align*}
    \{ T\in \mathfrak{S}': \mathfrak{V}_S T \in L^{p,q}_m(\mathbb{R}^{2d};\mathcal{HS}) \} \cong \{\Psi\in L^{p,q}_m(\mathbb{R}^{2d};\mathcal{HS}): \Psi = \Psi \natural \mathfrak{V}_S S\},
\end{align*}
for any $S\in\mathfrak{A}_v$ with $\|S\|_{\mathcal{HS}}=1$. 

\end{proof}

\section{Atomic Decomposition}\label{atomdecomp}
Coorbit spaces were introduced as a means of giving atomic decompositions with respect to unitary representations, and are fundamental to the field of time-frequency analysis for this reason. It is therefore natural, once one has such spaces, to consider the resulting discretisation. In particular we are interested in the discretisation of the identity
\begin{align*}
    T = \mathfrak{V}_S^* \mathfrak{V}_S T
\end{align*}
for $T\in\mathfrak{M}^{p,q}_m$ and $S\in\mathfrak{A}_v$, and 
the g-frame condition
\begin{align}
    A\|T\|_{\mathcal{HS}} \leq \sum_{\lambda \in \Lambda} \|S^*\pi(\lambda)^*T\|_{\mathcal{HS}} \leq B\|T\|_{\mathcal{HS}}
\end{align}
for some lattice $\Lambda\subset \mathbb{R}^{2d}$. We then proceed to consider the corresponding statement for operator modulation spaces, and interpret this as the statement that for $T$ an operator with poor time-frequency concentration, in some $\mathfrak{M}^{p,q}_m$ for large $p,q$, we can nonetheless decompose $T$ into well localised operators in the above manner.

In \cite{skrett21}, a similar problem was considered, of the conditions for which decompositions of functions $\psi\in M^p_m(\mathbb{R}^d)$ of the form
\begin{align*}
    \sum_{\lambda\in\Lambda} \alpha_{\lambda}(SS^*)\psi
\end{align*}
converge in a given norm. In that work the primary operators of interest were those of the form $S\in M^1_v(\mathbb{R}^d) \otimes M^1_v(\mathbb{R}^d)$, although as discussed in Remark 7.8 of that work, the same results hold for operators $S = \sum_n f_n \otimes g_n$ where $\{g_n\}_n$ is an orthonormal system in $L^2(\mathbb{R}^{2d})$, and $\{f_n\}_n\subset M^1_v$, with the condition $\sum_n \|f_n\|_{M^1_v} < \infty$.

With the twisted convolution structure of our coorbit spaces already in place, atomic decomposition results can be derived in an almost identical way to the function case, as presented in chapter 12 of \cite{grochenigtfa}, with some slight changes to accommodate the operator setting, based on the Wiener Amalgam spaces defined in \cref{Wieneramalg}. We present the proofs here for completeness. %
\todo{Recall def of Wiener spaces?}
\begin{lemma} \label{amalgconv}
Given $G\in W(L^1_v(\mathbb{R}^{2d};\mathcal{HS}))$ and $F\in L^{p,q}_m(\mathbb{R}^{2d};\mathcal{HS})$ continuous functions, where $m$ is a $v$-moderate weight, we have
\begin{align*}
    \|F\natural G\|_{W(L^{p,q}_m)} \leq C \|F\|_{L^{p,q}_m} \|G\|_{W(L^1_v)}.
\end{align*}
\end{lemma}
\begin{proof}
We construct the function $G_s(z)=\sum_{\lambda\in\Lambda} S_{\lambda}\cdot \chi_{\Omega_{\lambda}}(z)$, where $\Omega_{\lambda} = \lambda + [0,1]^{2d}$ and $S_{\lambda}$ is a value of $G$ in $\Omega_{\lambda}$ which maximises $\|G(z)\|_{\mathcal{HS}}$, which exists since $G$ is assumed to be continuous. Then $\|G(z)\|_{\mathcal{HS}} \leq \|G_s(z)\|_{\mathcal{HS}}$ and $\|G\|_{W(L^1_v)} = \|G_s\|_{W(L^1_v)}$. We then have
\begin{align*}
    \|F\natural G\|_{W(L^{p,q}_m)} &\leq \sum_{\lambda\in\Lambda} \|S_{\lambda}\|_{\mathcal{HS}} \| F\natural T_{\lambda} \chi_{\Omega_0}\|_{W(L^{p,q}_m)} \\
    &\leq \sum_{\lambda\in\Lambda} v(\lambda)\|S_{\lambda}\|_{\mathcal{HS}} \| F\natural \chi_{\Omega_0}\|_{W(L^{p,q}_m)} \\
    &= \| F\natural \chi_{\Omega_0}\|_{W(L^{p,q}_m)} \|G\|_{W(L^1_v)}.
\end{align*}
We abuse notation here by taking the twisted convolution of a vector valued and scalar valued function, but we interpret $F\natural \chi_{\Omega_0}(z)$ simply as $\int_{z-\Omega_0} F(z')c(z,z')dz'$. We also comment that while $S_{\lambda}$ may not be the value of $G$ maximising $F\natural G$, it nonetheless provides the upper bound in the first line. We consider the sequence
\begin{align*}
    a_{\lambda} &= \esssup_{z\in\Omega_{\lambda}} \|F\natural \chi_{\Omega_0}(z+\lambda)\|_{\mathcal{HS}} \\
    &\leq \esssup_{z\in\Omega_{\lambda}} \int_{z + \lambda-\Omega_0} \|F(z')\|_{\mathcal{HS}}dz' \\
    &\leq \int_{\lambda-\Tilde{\Omega}_0} \|F(z')\|_{\mathcal{HS}}dz' \\
    &= (\|F\|_{\mathcal{HS}}*\chi_{\Tilde{\Omega}_0})(z + \lambda)
\end{align*}
where $\Tilde{\Omega}_0 = [-1,1]^{2d}$, and $\|F\|_{\mathcal{HS}}$ is considered a scalar valued $L^{p,q}_m$ function. Moreover, we see that  $a_{\lambda}\chi_{\Omega_{\lambda}}(z) \leq (\|F\|_{\mathcal{HS}}*\chi_{\Check{\Omega}_0})(\lambda+z)$ for $z\in [0,1]^{2d}$, where here $\Check{\Omega}_0=[-2,2]^{2d}$, and so 
\begin{align*}
    \sum_{\lambda\in\Lambda} a_{\lambda}\chi_{\Omega_{\lambda}}(z) \leq (\|F\|_{\mathcal{HS}}*\chi_{\Check{\Omega}_0})(z).
\end{align*}
Finally we conclude
\begin{align*}
    \| F\natural \chi_{\Omega_0}\|_{W(L^{p,q}_m)} &= \|a\|_{l^{p,q}_m} \\
    &\leq C^\prime\|\sum_{\lambda\in\Lambda} a_{\lambda}\chi_{\Omega_{\lambda}} \|_{L^{p,q}_m} \\
    &\leq C^\prime\| \|F\|_{\mathcal{HS}}*\chi_{\Check{\Omega}_0} \|_{L^{p,q}_m} \\
    &\leq C\|F\|_{L^{p,q}_m} \|\chi_{\Check{\Omega}_0}\|_{L^1_v}
\end{align*}
where we have used Young's inequality for mixed norm spaces in the last line. The claim follows.
\end{proof}
\todo{Since, for $\varphi_0$....we have}In the function case, $V_{\varphi_0} \varphi_0 \in W(L^1_v(\mathbb{R}^{2d}))$, from which it follows that $\mathfrak{V}_{S_0} S_0 \in W(L^1_v(\mathbb{R}^{2d};\mathcal{HS}))$, since $\|\mathfrak{V}_{S_0} S_0(z)\|_{\mathcal{HS}}=|V_{\varphi_0} \varphi_0(z)|$.
\begin{corollary}\label{l1amalg}\label{atomicwiener}
    If $T\in \mathfrak{M}^{p,q}_m$ and $S\in\mathfrak{A}_v$, then $\mathfrak{V}_S T \in W(L^{p,q}_m(\mathbb{R}^{2d};\mathcal{HS}))$.
\end{corollary}
\begin{proof}
    From \cref{amalgconv}, for any $S\in \mathfrak{A}_v$, $\mathfrak{V}_{S_0} S\in W(L^1_v(\mathbb{R}^{2d};\mathcal{HS}))$. By then considering $\mathfrak{V}_{S_0} S \natural \mathfrak{V}_S S_0$ in the equation \eqref{twistedconv}, it follows that $\mathfrak{V}_S S\in W(L^1_v(\mathbb{R}^{2d};\mathcal{HS}))$ again by \cref{amalgconv}. The corollary then follows for general $T\in \mathfrak{M}^{p,q}_m$ from \Cref{youngsineq} and \cref{amalgconv}.
\end{proof}
With these preliminaries the boundedness of the analysis operator follows painlessly as in the function case presented in \cite{grochenigtfa};
\begin{proposition}
    For $S\in\mathfrak{A}_v$, the analysis operator $C_S: \mathfrak{M}^{p,q}_{{m}}\to l^{p,q}_{\Tilde{m}}(\Lambda,\mathcal{HS})$, defined by 
    \begin{align*}
        C_S(T) = \{S^*\pi(\lambda)^*T\}_{\lambda\in\Lambda},
    \end{align*}
    is a bounded operator with norm 
    \begin{align*}
        \|C_S\| \leq C\|\mathfrak{V}_S S\|_{W(L^1_v)},
    \end{align*}
    where the constant $C$ depends only on the lattice $\Lambda$ and weight $v(z)$
\end{proposition}

\begin{proof}
By \cref{atomicwiener}, $\mathfrak{V}_S S \in W(L^1_v(\mathbb{R}^{2d};\mathcal{HS}))$. Since $\mathfrak{V}_S T$ is continuous, we have from \cref{atomicwiener} and \Cref{wieneramalg} that
\begin{align*}
     \|C_S(T)\|_{l^{p,q}_{\Tilde{m}}(\Lambda,\mathcal{HS})} &= \|\mathfrak{V}_S T|_{\Lambda}\|_{l^{p,q}_{\Tilde{m}}(\Lambda,\mathcal{HS})} \\
     &\leq C' \|\mathfrak{V}_S T\|_{W(L^{p,q}_m)} \\
     &\leq C \|\mathfrak{V}_S S\|_{W(L^1_v)} \|T\|_{\mathfrak{M}^{p,q}_m}.
\end{align*}
\end{proof}
On the other hand, we find that the synthesis operator is similarly bounded, again in the same manner as the function case of \cite{grochenigtfa}:
\begin{proposition}
For $S\in\mathfrak{A}_v$, the synthesis operator $D_S: l^{p,q}_{\Tilde{m}}(\Lambda,\mathcal{HS}) \to \mathfrak{M}^{p,q}_{{m}}$, defined by 
    \begin{align*}
        D_S((T_{\lambda})_{\lambda\in\Lambda}) = \sum_{\lambda\in\Lambda} \pi(\lambda)S T_{\lambda}
    \end{align*}
is a bounded operator with norm
\begin{align*}
    \|D_S\| \leq C \|\mathfrak{V}_S S\|_{W(L^1_v)}.
\end{align*}
Convergence is interpreted to be unconditional for $p,q< \infty$, otherwise weak*, and the constant $C$ depends only on the lattice $\Lambda$ and weight $v(z)$.
\end{proposition}
\begin{proof}
We are required to show that $\mathfrak{V}_S D_S ((T_{\lambda})_{\lambda\in\Lambda}) \in L^{p,q}_m$. By definition;
\begin{align*}
    \|\mathfrak{V}_S D_S ((T_{\lambda})_{\lambda\in\Lambda})(z)\|_{\mathcal{HS}} &= \|\sum_{\lambda\in\Lambda} S^*\pi(z)^* \pi(\lambda)S T_{\lambda}\|_{\mathcal{HS}} \\
    &\leq \sum_{\lambda\in\Lambda} \|\mathfrak{V}_S S(z-\lambda)\|_{\mathcal{HS}}\|T_{\lambda}\|_{\mathcal{HS}}
\end{align*}
We have from \Cref{l1amalg} that $\mathfrak{V}_S S\in W(L^1_v)$,
so we denote once more
$G(z) = \sum_{\lambda \in \Lambda} S_{\lambda}\cdot\xi_{\Omega_{\lambda}} (z)$, where $S_{\lambda}$ 
is the value of $\mathfrak{V}_S S$ maximising the norm over $\lambda+[0,1]^{2d}$ as in \Cref{amalgconv}. We see then that the $L^{p,q}_m$ norm is bounded (up to a constant) by the discrete $l^{p,q}_{\Tilde{m}}$ norm of the convolution of sequences $s=(\|S_{\lambda}\|_{\mathcal{HS}})$ and $t=(\|T_{\lambda}\|_{\mathcal{HS}})$, and hence
\begin{align*}
    \|\mathfrak{V}_S D_S ((T_{\lambda})_{\lambda\in\Lambda})\|_{L^{p,q}_m} &\leq C' \|s*t\|_{l^{p,q}_{\Tilde{m}}} \\
    &\leq C'' \|s\|_{l^1_{\Tilde{v}}}\|t\|_{l^{p,q}_{\Tilde{m}}},
\end{align*}
and since $\|\mathfrak{V}_S S\|_{W(L^1_v)} = \|s\|_{l^1_{\Tilde{v}}}$, it follows that
\begin{align*}
    \|D_S\| \leq C \|\mathfrak{V}_S S\|_{W(L^1_v)}.
\end{align*}
Unconditional convergence for $p,q<\infty$ follows from the boundedness of $D_S$, since finite sequences are dense in $l^{p,q}_m$. For the case $p=\infty$ or $q=\infty$, the same fact can be used for the series $\langle R, \sum_{\lambda} \pi(\lambda)S T_{\lambda}\rangle_{\mathfrak{M}^1_{v},\mathfrak{M}^{\infty}_{1/v}}$, for all $R\in\mathfrak{M}^1_v$.
\end{proof}

\begin{corollary}
    Given $S, R \in\mathfrak{A}_v$, the frame operator $\mathfrak{O}_{S,R}:= D_S C_R$ is a bounded operator on $\mathfrak{M}^{p,q}_m$ for all $1\leq p,q\leq \infty$ and $v$-moderate weights $m$, with operator norm
    \begin{align*}
        \|\mathfrak{O}_S\| \leq C \|\mathfrak{V}_S S\|_{W(L^1_v)} \|\mathfrak{V}_R R\|_{W(L^1_v)}.
    \end{align*}
\end{corollary}
As a final corollary, we see that Gabor g-frames for operators in $\mathfrak{A}_v$ generate equivalent norms on $\mathfrak{M}^{p,q}_m$. We note that while stated for general $S,R$, we can always consider the canonical dual frame $\{S^*\pi(\lambda)^*\mathfrak{O}^{-1}\}$ given a Gabor g-frame $S\in\mathfrak{A}_v$.
\begin{corollary}
    If $S,R\in\mathfrak{A}_v$ are dual Gabor g-frames, so $\mathfrak{O}_{S,R} = I_{\mathcal{HS}}$, then $\mathfrak{O}_{S,R} = \mathfrak{O}_{R,S} = I_{\mathfrak{M}^{p,q}_m}$ where the sum is unconditional for all $1\leq p,q < \infty$, and weak* otherwise. Furthermore, there are constants $A,B$ such that
    \begin{align*}
        A\|T\|_{\mathfrak{M}^{p,q}_m} \leq \|S^*\pi(\lambda)^* T\|_{l^{p,q}_m} \leq B \|T\|_{\mathfrak{M}^{p,q}_m}
    \end{align*}
    (and similarly for $R$).
\end{corollary}

\begin{remark}
\normalfont It would be nice to be able to decompose an operator $T\in\mathcal{HS}$ solely in terms of $\alpha_{\lambda}$ shifts of some window $S\in\mathcal{HS}$, that is, in the form $T=\sum_{\lambda} c_{\lambda} \alpha_{\lambda}(S)$. In general this is impossible, since there does not exist and operator $S\in\mathcal{HS}$ and lattice $\Lambda$ such that the linear span of $\{\alpha_{\lambda}(S)\}_{\lambda\in\Lambda}$ is dense in $\mathcal{HS}$ (Proposition 7.2, \cite{skrett20b}). This is roughly due to the fact that one must have control of both sides of the tensor product $L^2(\mathbb{R}^d)\otimes L^2(\mathbb{R}^d)$, such as in the case presented in \cite{balasz2008} wherein frames for $\mathcal{HS}$ are generated by two frames for $L^2(\mathbb{R}^d)$. Restricting to the set of positive $\mathcal{HS}$ operators, shifts of $S=\varphi_0 \otimes \varphi_0$ span a dense subset for any lattice for which $\varphi_0$ is a frame for $L^2(\mathbb{R}^d)$. However a positive $S$ may fail to be generate a dense subset for greater rank even when constructed from functions forming a multi-window Gabor frame, since for example even a rank-one $T=f\otimes f$ can only be reconstructed if the coefficients for all functions making up the multi-window Gabor frame in S are the coincide for $f$, ie $f=\sum_{\lambda} c_{\lambda} \sum_j \pi(\lambda)g_j$ where $S=\sum_j g_j \otimes g_j$. This inability to decompose $\mathcal{HS}$ operators solely as $\alpha_{\lambda}$ shifts of some window, shows the necessity of g-frames in our setting.
\end{remark}

\subsection{Modulation Space Characterisation by Localisation Operators}
In \cite{dor06} and \cite{dor11}, the authors consider the characterisation of modulation spaces by g-frames of translated localisation operators, initially for the Gelfand triple $(M^1(\mathbb{R}^d),L^2(\mathbb{R}^d),M^{\infty}(\mathbb{R}^d))$, and later for general $M^p_m(\mathbb{R}^d)$:
\begin{theorem}[Theorem 8, \cite{dor11}]
Let $\varphi\in M^1_v(\mathbb{R}^d)$ be non-zero and $h\in L^1_v(\mathbb{R}^{2d})$ be some non-negative symbol satisfying
\begin{align}\label{locopframecond}
    A \leq \sum_{\lambda\in\Lambda} h(z-\lambda) \leq B
\end{align}
for positive constants $A,B$, and almost all $z\in\mathbb{R}^{2d}$. Then for every $v$-moderate weight $m$ and $1\leq p < \infty$, the function $f\in M^{\infty}_{1/v}(\mathbb{R}^d)$ belongs to $M^p_m(\mathbb{R}^d)$ if and only if
\begin{align*}
    (\sum_{\lambda\in\Lambda} \|A_h^{\varphi} \pi(\lambda)^*f\|^p_{L^2}m(\lambda))^{1/p} < \infty,
\end{align*}
where $A_h^{\varphi}:f \mapsto V_{\varphi}^*(h\cdot V_{\varphi}f)$ is the localisation operator with symbol $h$. In this case the left hand side is an equivalent norm to $\|\cdot\|_{M^p_m}$. Similarly for $p=\infty$;
\begin{align*}
    \|f\|_{M^{\infty}_m} \asymp \sup_{\lambda\in\Lambda} \|A_h^{\varphi} \pi(\lambda)^*f\|_{L^2}m(\lambda)
\end{align*}
\end{theorem}
 We note in particular that condition \eqref{locopframecond} gives criteria for $A_h^{\varphi}$ to generate a Gabor g-frame, which one sees by considering $p=2$, $m\equiv 1$. We can characterise operator coorbit spaces similarly. We confirm that we can consider Gabor g-frames on $L^2(\mathbb{R}^d)$ in the operator setting:
\begin{proposition}
    $S$ generates a Gabor g-frame on $L^2(\mathbb{R}^d)$ if and only if $S$ generates a Gabor g-frame on $\mathcal{HS}$, with the same frame constants.
\end{proposition}
\begin{proof}
Assume $S$ generates a Gabor g-frame on $L^2(\mathbb{R}^d)$, that is
\begin{align}
    A\|f\|_{L^2}^2 \leq \sum_{\lambda\in\Lambda} \|S^*\pi(\lambda)^*f\|_{L^2}^2 \leq B\|f\|_{L^2}^2
\end{align}
for all $f\in L^2(\mathbb{R}^d)$. Then any $T\in\mathcal{HS}$ can be decomposed as $T=\sum_n f_n \otimes e_n$ for some orthonormal set $\{e_n\}_n$, and the trace taken with respect to $\{e_n\}_n$ (which can be extended to an orthonormal basis if T is not full rank):
\begin{align*}
    \sum_{\lambda\in\Lambda}\|S^*\pi(\lambda)^*T\|_{\mathcal{HS}}^2 &= \sum_{\lambda\in\Lambda} \sum_{n\in\mathbb{N}} \langle S^*\pi(\lambda)^*T e_n, S^*\pi(\lambda)^*T e_n\rangle_{L^2} \\
    &= \sum_{\lambda\in\Lambda} \sum_{n\in\mathbb{N}} \|S^*\pi(\lambda)^*f_n\|_{L^2}^2.
\end{align*}
The Gabor g-frame condition on $L^2(\mathbb{R}^d)$ then gives
\begin{align*}
    \sum_{n} A\|f_n\|_{L^2}^2 \leq \sum_{n\in\mathbb{N}} \sum_{\lambda\in\Lambda} \|S^*\pi(\lambda)^*f_n\|_{L^2}^2 \leq \sum_{n} B\|f_n\|_{L^2}^2,
\end{align*}
and so 
\begin{align}
    A\|T\|_{\mathcal{HS}}^2 \leq \sum_{\lambda\in\Lambda}\|S^*\pi(\lambda)^*T\|_{\mathcal{HS}}^2 \leq B\|T\|_{\mathcal{HS}}^2.
\end{align}
The opposite direction follows by the same expansion with a rank one operator.
\end{proof}
The following characterisation then uses Proposition 7.14 of \cite{skrett21}, which states that given $h\in L^1_{v^2}(\mathbb{R}^d)$, $A_h^{\varphi}\in M^1_v(\mathbb{R}^d)\otimes M^1_v(\mathbb{R}^d)$, which in particular tells us $A_h^{\varphi}\in\mathfrak{A}_v$.
\begin{corollary}
    Given $h\in L^1_{v^2}(\mathbb{R}^{2d})$ satisfying \eqref{locopframecond} and some $v$-moderate $m$, the operator $T\in\mathfrak{M}^{\infty}_{1/v}$ belongs to $\mathfrak{M}^{p,q}_m$ if and only if 
    \begin{align*}
        \big\{ A_{\overline{h}}^{\varphi} \pi(\lambda)^* T \big\}_{\lambda\in\Lambda} \in l^{p,q}_m(\Lambda;\mathcal{HS}).
    \end{align*}
\end{corollary}
In a similar manner to \Cref{tfintuition}, this corollary supports the intuition of the $\mathfrak{M}^{p,q}_m$ condition measuring the time-frequency decay in the operator sense. We often consider localisation operators with symbol $h$ having essential support concentrated in some domain $\Omega$, such as the characteristic function $\chi_{\Omega}$. Hence $A_h^{\varphi}$ can be seen as measuring the time frequency concentration of a function in $\Omega$. With this intuition we can consider $ A_{\overline{h}}^{\varphi}\pi(\lambda)^*T$ as measuring how much $T$ concentrates a function to some domain $\Omega + \lambda$, and thus we interpret the sum over $\lambda$ as a measure of the extent to which $T$ spreads out functions in the time-frequency plane. 

\section{Final Remarks}
This paper introduces an operator STFT, a novel concept bearing potential both in theoretical settings and applications to data analysis and quantum harmonic analysis. The main results of the paper arise from representing an ensemble of data points or signals with respect to a joint time-frequency representation, which captures correlations between data points in the time-frequency plane (\Cref{dataopex}). We show that the operator STFT has many of the familiar properties of the function STFT, and in particular that the spaces produced by the operator STFT with a fixed window are reproducing kernel Hilbert spaces. The Toeplitz operators associated with such spaces are the mixed-state localisation operators, an observation, that  supports the notion of the operator STFT extending the function STFT to an appropriate object for quantum harmonic analysis (\Cref{repkernstruc}). From a functional data analysis point of view, stable representations of continuous data is the ideal, and so having a reproducing structure when analysing data sets ensures stability with respect to noise and small perturbations in the incoming data. Furthermore, by extending the spaces of operators one considers, we are able to define coorbit spaces for operators. It turns out, that  we thus obtain Banach spaces of operators behaving analogously to the function coorbit spaces, with regards to duality, the equivalence of window in the definition, and even the correspondence principle for the coorbit spaces (\Cref{opcoorbit}). If we interpret function coorbit spaces as those functions appropriately concentrated in a region of the time-frequency plane, we can consider the operator coorbit spaces as those operators which \textit{act} on functions in a concentrated region of the time-frequency plane (\cref{tfintuition}). These coorbit spaces of operators turn out to have the remarkable property of decomposition via Gabor g-frames, which says that given an operator in a coorbit space, we can write the operator as a sum of translations of well localised operators (\Cref{atomdecomp}).





\pagebreak
\bibliography{refs}
\bibliographystyle{abbrv}

\Addresses

\end{document}